\documentclass{amsart}       
\usepackage{graphicx}
\usepackage{color}
\usepackage{amsmath} 
\usepackage{amsthm}
\usepackage{enumitem}
\usepackage{tikz}
\newtheorem{theorem}{Theorem}[section]
\newtheorem{corollary}{Corollary}[section]
\newtheorem{lemma}{Lemma}[section]
\theoremstyle{remark}
\newtheorem{remark}{Remark}[section]
\newcommand{\eps}{\varepsilon}
\newcommand{\ESMM}[1]{E^{#1}_{\mathrm{sMM}}}
\newcommand{\ESMMb}[2]{E^{#1,#2}_{\mathrm{sMM}}}
\newcommand{\EKWC}[1]{E^{#1}_{\mathrm{KWC}}}
\newcommand{\dist}{\operatorname{dist}}
\newcommand{\graphto}{\xrightarrow{g}}
\DeclareMathOperator{\cl}{cl}
\newcommand{\abs}[1]{\left|#1\right|}
\newcommand{\odfrac}[2]{\frac{\mathrm{d}#1}{\mathrm{d}#2}}
\newcommand{\id}{\,\mathrm{d}}
\DeclareMathOperator*{\liminfstar}{lim\,inf_{\lower 3pt \hbox to 3pt{\tiny $\ast$}}}
\DeclareMathOperator*{\limsupstar}{lim\,sup^*}

\newcommand{\ringM}{\mathring{M}}
\newcommand{\graph}{\operatorname{graph}}

\numberwithin{equation}{section}
\begin{document}

\title[A finer singular limit]{A finer singular limit of a single-well Modica--Mortola functional
 and its applications to the Kobayashi--Warren--Carter energy}
\author{Yoshikazu~Giga}
\address{Graduate School of Mathematical Sciences \\
The University of Tokyo \\
3-8-1 Komaba Meguro-ku Tokyo 153-8914, Japan}
\email{labgiga@ms.u-tokyo.ac.jp}

\author{Jun~Okamoto}
\address{Graduate School of Mathematical Sciences \\
The University of Tokyo \\
3-8-1 Komaba Meguro-ku Tokyo 153-8914, Japan}
\email{okamoto@ms.u-tokyo.ac.jp}

\author{Masaaki~Uesaka}
\address{Arithmer, Inc. \\
1-6-1 Roppongi Minato-ku Tokyo 106-6040 \\
and Graduate School of Mathematical Sciences \\
The University of Tokyo \\
3-8-1 Komaba Meguro-ku Tokyo 153-8914, Japan, \\
}
\email{masaaki.uesaka@arithmer.co.jp \\ muesaka@ms.u-tokyo.ac.jp}
\maketitle

\begin{abstract}
An explicit representation of the Gamma limit of a single-well Modica--Mortola functional is given for one-dimensional space under the graph convergence which is finer than conventional $L^1$-convergence or convergence in measure.
 As an application, an explicit representation of a singular limit of the Kobayashi--Warren--Carter energy, which is popular in materials science, is given.
 Some compactness under the graph convergence is also established.
 Such formulas as well as compactness is useful to characterize the limit of minimizers the Kobayashi--Warren--Carter energy.
 To characterize the Gamma limit under the graph convergence, a new idea which is especially useful for one-dimensional problem is introduced.
 It is a change of parameter of the variable by arc-length parameter of its graph, which is called unfolding by the arc-length parameter in this paper.
\end{abstract}

\section{Introduction} \label{intro} 

In this paper, we are interested in a singular limit called the Gamma limit of a single-well Modica--Mortola functional under the graph convergence, the convergence with respect to the Hausdorff distance of graphs, which is finer than conventional $L^1$-convergence or convergence in measure.
 A single-well Modica--Mortola functional is introduced by Ambrosio and Tortorelli~\cite{AT1,AT2} to approximate the Mumford--Shah functional~\cite{MS}.
 A typical explicit form of their functional now called the Ambrosio--Tortorelli functional is
\[
	\mathcal{E}^\eps(u,v) 
	:= \sigma \int_\Omega v^2 \abs{\nabla u}^2 \id x
	+ \lambda \int_\Omega (u-g)^2 \id x
	+ E^\eps(v)
\]
with small parameter $\eps > 0$, where $E^\eps$ is a single-well Modica--Mortola functional of the form
\[
	E^\eps(v) 
	:= \frac{1}{2\eps} \int_\Omega (v-1)^2 \id x
	+ \frac{\eps}{2} \int_\Omega \abs{\nabla v}^2 \id x.
\]
Here $g$ is a given function defined in a bounded domain $\Omega$ in $\mathbf{R}^n$ and $\sigma \geq 0$, $\lambda \geq 0$ are a given parameters.
 The potential energy part $(v-1)^2$ is a single-well potential.
 If it is replaced by a double-well potential like $(v^2-1)^2$, the corresponding energy $E^\eps$ well approximates (a constant multiple of) the surface area of the interface and this observation went back to Modica and Mortola~\cite{MM1,MM2}.
 Even for the single-well potential if $v$ is close to zero around some interface then it is expected that $E^\eps$ still approximates {the surface area of the interface}.
 This observation enables us to prove that
 for $\sigma>0$, the Gamma limit of $\mathcal{E}^\eps(u,v)$ in the convergence in measure is a Mumford--Shah functional; see \cite{AT1,AT2,FL}.

If $E^\eps(v_\eps)$ is bounded for small $\eps > 0$, then it is rather clear that $v_\eps\to 1$ in $L^1$ as $\eps\to 0$,
so that $v_{\eps'}\to 1$ almost everywhere by taking a suitable subsequence.
 Therefore, it seems natural to consider the Gamma convergence in $L^1$-sense.
 However, if one considers
\begin{equation}\label{eq:Eepsb}
	E^\eps_b(v) = E^\eps(v) + b v(0)^2
\end{equation}
for $b>0$, where $\Omega=(-1,1)$, then we see $L^1$-convergence is too weak because in the limit stage, the effect of the term involving $b$ is invisible but this should be counted.

To illustrate the point, we calculate the unique minimizer $w_\eps$ of $E^\eps_b(v)$, that is,
\[
	E^\eps_b(w_\eps) 
	= \min \left\{ E^\eps_b(v) \mid v \in H^1(-1,1) \right\}.
\]
This is strict convex problem so that the minimizer exists and unique.
 Moreover, its Euler--Lagrange equation is linear.
 A simple manipulation shows that the minimizer of $E^\eps_b$ with the Neumann boundary conditions $w_\eps'(\pm 1) = 0$ is given by
\[
	w_{\eps}(x)=1+\frac{b\left(-e^{-\frac{2}{\eps}}-1\right)}{1- e^{-\frac{4}{\eps}}+b\left(1+e^{-\frac{2}{\eps}}\right)^{2}} e^{-\frac{|x|}{\eps}}+\frac{ b\left(-e^{-\frac{2}{\eps}}-e^{-\frac{4}{\eps}}\right)}{1- e^{-\frac{4}{\eps}}+b\left(1+e^{-\frac{2}{\eps}}\right)^{2}} e^{\frac{|x|}{\eps}}.
\]
It converges to $1$ locally uniformly outside zero but
\[
	\lim_{\eps\to 0} w_\eps(0) = \frac{1}{1+b} > 0
\]
and
\[
    \lim_{\eps \to 0} E^\eps_b(w_\eps) = \left(\frac{b}{1+b}\right)^2 < b.
\]
Since $E^\eps_b(1)=b$ for any $\eps > 0$, the information that $w_\eps(x) \to 1$ almost everywhere is insufficient to identify the behavior of minimizers $w_\eps$.
\begin{figure}
    \centering
    \includegraphics[width=\linewidth]{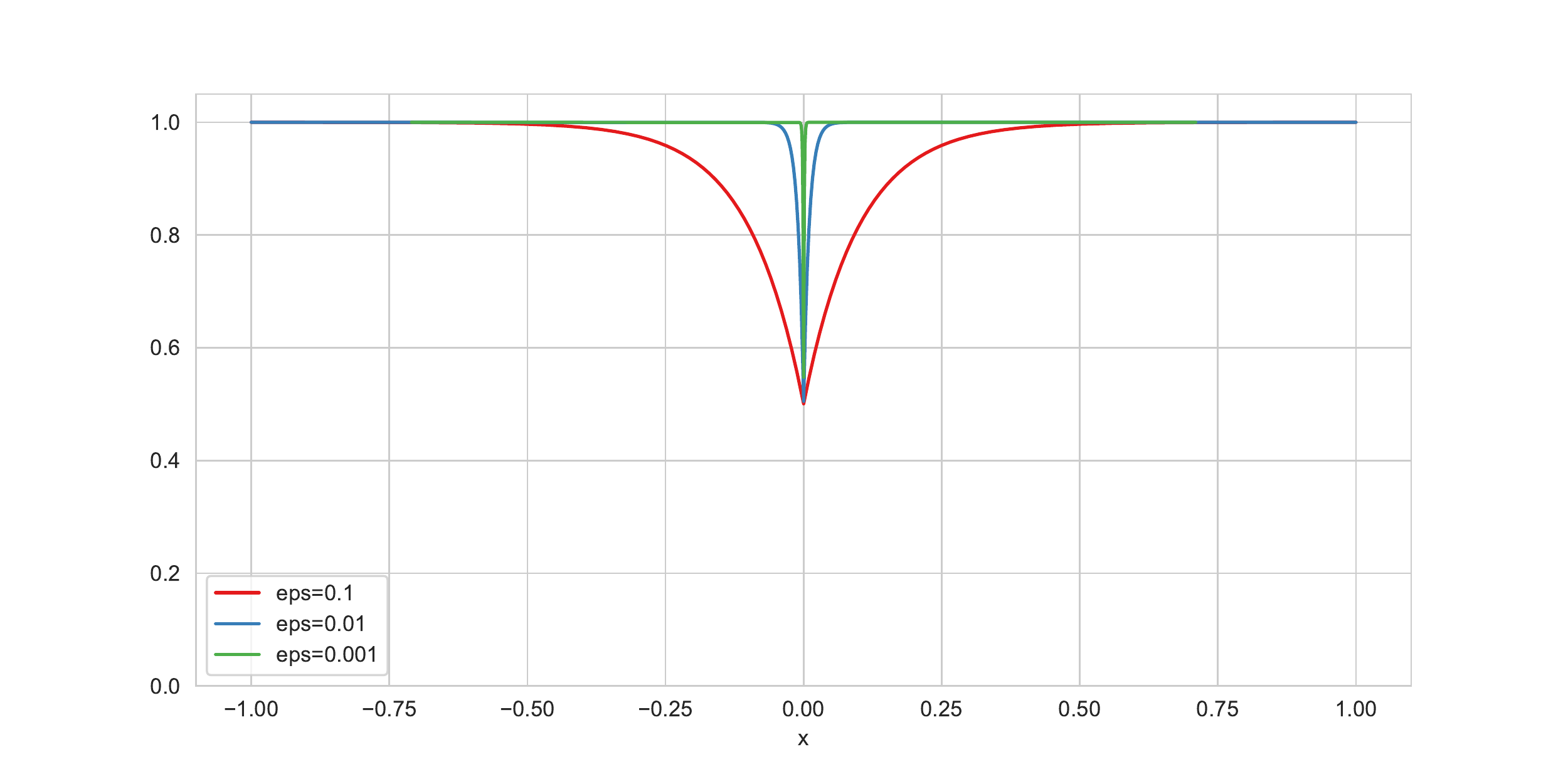}
    \caption{The graphs of $w_\eps$ as the minimizers of $E^\eps_b$ defined by~\eqref{eq:Eepsb} when $b=1$ and $\eps = 10^{-1}, 10^{-2}, 10^{-3}$.}
    \label{fig:graph_Eepsb}
\end{figure}

We show the graph of $w_\eps$ for several $\eps > 0$ in Figure~\ref{fig:graph_Eepsb}.
We see that the graph of $w_\eps$ is dropping sharply at $x=0$ and its sharpness increases as $\eps \to 0$.
Hence, it is natural to consider the graph convergence of $w_\eps$ and its limit is a set-valued function $\Xi$ so that $\Xi(x)=\{1\}$ for $x\neq 0$ and $\Xi(0) = \left[1/(1+b),1\right]$.

Our first goal is to give an explicit representation formula for the Gamma limit of $E^\eps_b$
under the graph convergence as well as compactness.
We discuss such problems only in one-dimensional domain
since the problem is already complicated.
The graph convergence enables us to characterize the limit of above $w_\eps$
as a minimizer of the Gamma limit of $E^\eps_b$.
 
Our second goal is to give an explicit representation formula
for the Gamma limit of the Kobayashi--Warren--Carter energy.
A typical form of the energy is
\[
	\EKWC{\eps}(u,v) = \sigma \int_\Omega v^2 \abs{\nabla u} \id x + E^\eps(v).
\]
This energy is first proposed by \cite{KWC00,KWC98} to model motion of multi-phase problems in materials sciences.
 This energy looks similar to the Ambrosio--Tortorelli functional $\mathcal{E}^\eps$.
 It is obtained by inhomogenizing Dirichlet energy $\int|\nabla u|^2 \id x$ by putting weights $\int v^2 |\nabla u|^2 \id x$ with a single-well Modica--Mortola functional.
 By this observation, we call $\mathcal{E}^\eps$ an Ambrosio--Tortorelli inhomogenization of the Dirichlet energy when $\lambda=0$.
 From this point of view, the Kobayashi--Warren--Carter energy is interpreted as an Ambrosio--Tortorelli inhomogenization of the total variation.
 It turns out that natural topology for studying the limit of functionals as $\eps \to 0$ is quite different.
 
 For the Ambrosio--Tortorelli functional, it is enough to consider $L^1 \times L^1$ converges
 since $v_\eps(x) \to 1$ except finitely many points where $\liminfstar v_\eps(x)=0$
 if one assumes that $\mathcal{E}^\eps(u_\eps,v_\eps)$ is bounded and $u_\eps \to u$, $v_\eps \to v$ in $L^1$. (see \cite{AT1,AT2,FL}.) 
 Here $\liminfstar$ denotes the relaxed liminf and we shall give its definition in Section~\ref{sec:2}.
 For the Kobayashi--Warren--Carter energy, however, 
 the situation is quite different.
 Indeed, if one considers
\begin{eqnarray*}
	u(x) = \left\{ 
\begin{array}{lc}
	1, & 0<x<1 \\
	0, & -1<x<0,
\end{array} \right.
\end{eqnarray*}
then $\EKWC{\eps}(u,v)=E^\eps_\sigma(v)$ with $\Omega=(-1,1)$.
 Thus the natural convergence for $v$ must be in the graph convergence
 as we discussed before.
 Note that in our problem $v_\eps \to 1$ except countably many points
 and there $\liminfstar v_\eps$ may not be zero.
 One merit of the graph convergence is that
 it is very strong so when we consider the Gamma limit problem,
 we don't need to restrict ourselves in the space of special $BV$ functions as for the Ambrosio--Tortorelli functional.
 
Our first main result is a characterization of the Gamma limit of $E^\eps_b$ in the graph convergence (Theorem~\ref{G1}).
To show the Gamma convergence,
we need to prove the two types of inequalities often called liminf and limsup inequalities.
To show liminf inequality,
a key point is to study a general behavior near the set $\Sigma$ of all exponential points of the limit set-valued function $\Xi$;
here, we say a point $x$ is exceptional if $\Xi(x)$ is not a singleton.
To describe behavior near $\Sigma$, a conventional method is to find a suitable accumulating sequence as in~\cite[proof of Proposition 3.3]{FL}.
However, unfortunately, it seems that this argument does not apply to our setting,
since $\Sigma$ can be a countably infinite set.
Thus we are forced to introduce a new method to show liminf inequality.
When we study a absolutely continuous function $u^\eps$ on a bounded interval $I$,
that is, $u^\eps \in W^{1,1}(I)$,
we associate its unfolding $U^\eps$ by replacing the variable
by the arc-length parameter of the graph.
Namely, we set
\[
	U^\eps(s) = u^\eps \left( x^\eps(s) \right), \quad
	s \in J^\eps = s^\eps(I),
\]
where $x^\eps = x^\eps(s)$ in the inverse function of the arc-length parameter
\[
	s^\eps(x) = \int^x_0 \left( 1+ \left(u^\eps_x(z) \right)^2 \right)^{1/2} dz.
\]
If the total variation of $u^\eps$ is bounded,
then the length of $J^\eps$ is bounded as $\eps \to 0$.
The unfolding $U^\eps$ has several merits {compared with} the original one.
First, $\{U^\eps\}$ and $\{x^\eps\}$ are uniformly Lipschitz with constant $1$.
Second, the total variation of $U^\eps$ and $u^\eps$ is the same as expected.
It is easy to study the convergence as $\eps \to 0$ of unfolding $U^\eps$ compared with the original $u^\eps$.
Among other results, we are able to characterize the relaxed limits $\liminfstar u^\eps$, $\limsupstar u^\eps$ by the limit of $U^\eps$ and $x^\eps$.
We use this unfolding for $(v_\eps-1)^2/2$ in the case of $E^\eps_b$ to show liminf inequalities, 
where $\{v_\eps\}$ is a given sequence with a bound for $E^\eps_b(v_\eps)$.
The proof for limsup inequalities is not difficult although one has to be careful that there are countably many points where the limit of $v^\eps$ is not equal to one.
 
We also established a compactness under the graph convergence with a bound for $E^\eps_b$ (Theorem~\ref{C1}).
This can be easily proved by use of unfoldings.

Based on results on $E^\eps_b$,
we are able to prove the Gamma convergence of the Kobayashi--Warren--Carter energy  $\EKWC{\eps}$ under the graph convergence (Theorem~\ref{G2}).
 If $u$ is a piecewise constant but has a countably many jump points $\{a_\ell\}^\infty_{\ell=1}\subset\Omega$ with positive jump $\{b_\ell\}^\infty_{\ell=1}$,
 we see that
\[
	E^\eps_{\mathrm{KWC}}(u, v) = E^\eps(v) + \sigma \sum^\infty_{\ell=1} b_\ell v^2(a_\ell).
\]
The Gamma limit for such fixed $u$ is easily reduced to the results of $E^\eps_b$.
However, to establish liminf inequality for $\EKWC{\eps}$ for both $u_\eps$ and $v_\eps$,
we have to establish some lower estimate for a sequence $\int_\Omega v^2_\eps \abs{\nabla u_\eps} \id x$ as $\eps \to 0$,
which is an additional difficulty.
However, we still do not need to use {\textit{SBV}} space here.

The Gamma convergence problem of the Modica--Mortola functional,
which is the sum of Dirichlet type energy and potential energy was first studied by~\cite{MM1}.
Since then, there is a large number of works discussing the Gamma convergence.
However, the topology is either $L^1$ or convergence in measure.
In our Gamma limit, the topology is the graph convergence, which is finer than previous study.
In~\cite{MM2}, the $L^1$ Gamma limit of a double-well Modica--Mortola functional is characterized as a number of transition points in one-dimensional setting.
Later in~\cite{Mo,St}, it was extended to multi-dimensional setting and the limit is a constant multiple of the surface area of the transition interface.
This type of the Gamma convergence results as well as compactness is important to establish the convergence of local minimizer (\cite{KS}) as well as the global minimizer.
 However, the convergence of critical points are not in the framework of a general theory
 and a special treatment is necessary~\cite{HT}.
 The double-well Modica--Mortola functional is by now well studied even in the level of gradient flow {called the Allen--Cahn equation}.
 The limit $\eps \to 0$ is often called the sharp interface limit and the resulting flow is known as the mean curvature flow.
 For early stage of development of the theory, see~\cite{BK,XC,DS,ESS}.

A single-well Modica--Mortola functional is first used in~\cite{AT1} to approximate the Mumford--Shah functional.
 The Gamma limit of the Ambrosio--Tortorelli functional is by now well studied (\cite{AT1,AT2,FL}).
 However, convergence of critical points is studied only in one dimension (\cite{FLS}).
 The Ambrosio--Tortorelli type approximation is now used in various problems.
 In~\cite{FM}, the Ambrosio--Tortorelli type approximation is introduced to describe brittle fractures.
 Its evolution is also described in~\cite{G}.
 For the Steiner problem, such approximation as also proposed (\cite{LS}) and {its} Gamma limit is established (\cite{BLM}).
 However, all these problems the problem is closer to the Ambrosio--Tortorelli inhomogenization of the Dirichlet energy, not of the total variation.

For the Kobayashi--Warren--Carter energy, its gradient flow for fixed $\eps$ is somewhat studied.
 Note that the well-posedness itself is non-trivial because even if one assumes $v \equiv 1$,
 the gradient flow of $\EKWC{\eps}$ is the total variation flow
 and the definition of a solution itself is non trivial;
 see~\cite{KG}, for example.
 Apparently, there is no well-posedness result for the original system proposed by~\cite{KWC98,KWC00,KWC}.
 According to~\cite{KWC00}, its explicit form is
 \begin{align}
    \tau_1 v_t &= s \Delta v + (1 - v) - 2s v\abs{\nabla u}, \label{eq:KWC_grad_1}\\
    \tau_0 v^2 u_t &= s \operatorname{div} \left( v^2 \frac{\nabla u}{\abs{\nabla u}} \right), \label{eq:KWC_grad_2}
\end{align}
where $\tau_0$, $\tau_1$, $s$ are positive parameters.
This system is regarded as the gradient flow of $\EKWC{\eps}$ with $F(v) = (v - 1)^2$, $\eps = 1$, $\sigma = s$ with respect to a kind of weighted $L^2$ norm
whose weight depends on the solution.
If one replaces~\eqref{eq:KWC_grad_2} by
\[
 \tau_0 (v^2 +\delta) u_t = s \, \operatorname{div} \left( \bigl( v^2 +\delta' \bigr) \frac{\nabla u}{|\nabla u|} +\nu \nabla u \right)
\]
with $\delta > 0$, $\delta' \geq 0$, and $\nu \geq 0$ satisfying $\delta' +\nu > 0$,
then the studies of existence and large-time behavior of solutions are developed in \cite{IKY,MR3268865,MR3670006,MR3203495,MR3082861,SWY},
under homogeneous settings of boundary conditions. However, the uniqueness question is almost open,
and there is a few (only one) result \cite[Theorem 2.2]{IKY} for the one-dimensional solution,
under $ \nu > 0 $.
Meanwhile, the line of previous results can be extended to the studies of non-homogeneous cases of boundary conditions.
For instance, if we impose the non-homogeneous Dirichlet boundary condition for~\eqref{eq:KWC_grad_2},
then we can further observe various structural patterns of steady-state solutions,
under one-dimensional setting, two-dimensional radially-symmetric setting,
and so on (cf. \cite{mswD2020}). 

This paper is organized as follows.
 In Section~\ref{sec:2}, we recall notion of the graph convergence
 and states our main Gamma convergence results as well as compactness.
 In Section~\ref{sec:3}, we introduce notion of unfoldings.
 Section~\ref{sec:4} is devoted to the proof of the Gamma convergence of $E^\eps_b$ as well as the compactness in the graph convergence.
 Section~\ref{sec:5} is devoted to the proof of the Gamma convergence of the Kobayashi--Warren--Carter energy.

The authors are grateful to Professor Ken~Shirakawa
for letting us know his recent results before publication as well as development of researches on gradient flows of Kobayashi--Warren--Carter type energies.

\section{Singular limit under graph convergence} \label{sec:2} 

We first recall basic notion of set-valued functions; see~\cite{AF} for example.
 Let $(M,d_M)$ be a compact metric space.
 We consider a set-valued function $\Gamma$ defined in $M$ such that {$\Gamma(x)$ is a compact set in $\mathbf{R}$ for each $x \in M$.}
 If its $\graph\Gamma$ defined by
\[
	\graph\Gamma := \bigl\{ (x,y)\in M\times \mathbf{R}
	\bigm| y \in \Gamma(x),\ x \in M \bigr\}
\]
is closed, we say that $\Gamma$ is {\it upper semicontinuous}.
 Let $\mathcal{B}$ denote the totality of a bounded, upper semicontinuous set-valued functions.
 In other words,
\[
	\mathcal{B} := \left\{ \Gamma \mid \graph\Gamma\ \text{is compact in}\ M \times \mathbf{R} \right\}.
\]
For $\Gamma_1,\Gamma_2\in\mathcal{B}$, we set
\[
	d_g (\Gamma_1,\Gamma_2) := d_H (\graph\Gamma_1,\graph\Gamma_2),
\]
where $d_H$ denotes the Hausdorff distance of two sets in $M\times\mathbf{R}$.
 The Hausdorff distance $d_H$ is defined as usual:
\[
	d_H(A,B) := \max \left\{ \sup_{z \in A} \dist(z,B),\ \sup_{w \in B} \dist(w,A)\right\}
\]
for $A,B \subset M\times\mathbf{R}$, where
\[
	\dist(z,B) := \inf_{w \in B} \dist(z,w),\quad 
	\dist(z,w) := \left( d_M(z_1,w_1)^2 + |z_2-w_2|^2\right)^{1/2}
\]
for $z=(z_1,z_2)$ and $w=(w_1,w_2)$.
It is easy to see that $(\mathcal{B},d_g)$ is a complete metric space.
The convergence with respect to $d_g$ is called the {\it graph convergence}.

We next recall semi-convergent limit for sets.
 For a family of closed subsets $\{Z_\eps\}_{0<\eps<1}$ in $M\times\mathbf{R}$, we set

\begin{align*}
	\limsup_{\eps \to 0} Z_\eps &:= \bigcap_{\eps>0} \cl \left(\bigcup_{0<\delta<\eps}Z_\delta \right) \\
	\liminf_{\eps \to 0} Z_\eps &:= \cl \left(\bigcup_{\eps>0} \bigcap_{0<\delta<\eps}Z_\delta \right),
\end{align*}
where $\cl$ denotes the closure in $M\times\mathbf{R}$.
 These semi-limits can be defined for sequences {like $\{Z_j\}_{j=1}^{\infty}$} with trivial modification.
\begin{lemma} \label{SG1}
A sequence $\{\Gamma_j\}^\infty_{j=1}\subset\mathcal{B}$ converges to $\Gamma$ in the sense of the graph convergence if and only if
\[
	\limsup_{j\to\infty} \graph\Gamma_i
	= \liminf_{j\to\infty} \graph\Gamma_j = \graph\Gamma.
\]
\end{lemma}
\begin{proof}
Note that the Hausdorff convergence to $A$ for sequence $\{A_j\}^\infty_{j=1}$ of compact sets is equivalent to saying that
\begin{enumerate}
\item[(i)] for any $z \in A$, there is a sequence $z_j \in A_j$ such that $z_j \to z$ ($j \to \infty$) and 
\item[(i\!i)] if $w_j \in A_j$ converges to $w$, then $w \in A$.
\end{enumerate}
Since (i) and (i\!i) are equivalent to
\[
	\liminf_{j\to\infty} A_j \supset A,\ 
	\limsup_{j\to \infty} A_j \subset A,
\]
respectively, the Hausdorff convergence is equivalent to saying that
\[
	A = \liminf_{j\to\infty} A_j = \limsup_{j\to\infty} A_j.
\]
Thus the proof is complete.
\end{proof}

We next recall relaxed convergent limits of functions.
 Let $\{g_j\}$ be a sequence of real-valued function on $M$.
 {For $x \in M$,} We set
\begin{align*}
	\limsupstar_{j\to\infty} g_j(x) &:= \lim_{j\to\infty} \sup\bigl\{ g_k (y) \bigm| |y-x|<1/j,\ k \geq j \bigr\} \\
	\liminfstar_{j\to\infty} g_j(x) &:= \lim_{j\to\infty} \inf\bigl\{ g_k (y) \bigm| |y-x|<1/j,\ k \geq j \bigr\};
\end{align*}
see~\cite[Chapter 2]{Giga} for more detail.
By definition, the $\limsupstar w_j$ is upper semicontinuous and $\liminfstar w_j$ is lower semicontinuous.

Let $C(M)$ be the Banach space of all continuous real-valued functions on $M$  equipped with the norm $\|f\|_\infty=\sup_{x \in M}\left|f(x)\right|$, $f \in C(M)$.
 For $g \in C(M)$, we associate a set-valued function $\Gamma_g$ such that $\Gamma_g(x)=\left\{g(x)\right\}$ for $x \in M$.
 Clearly, $\Gamma_g \in \mathcal{B}$.
\begin{lemma} \label{SG2} 
Let $\left\{g_j\right\}^\infty_{j=1} \subset C(M)$ be a bounded sequence.
 Then the semi-limit $\Gamma_+=\limsup_{j\to\infty}\Gamma_{g_j}$ still belongs to $\mathcal{B}$.
 Let $K$ be the set-valued function of the form
\[
	K(x) := \left\{ y \in \mathbf{R} \Bigm| \liminfstar_{j\to \infty} g_j(x) \leq y \leq \limsupstar_{j\to\infty} g_j(x) \right\}.
\]
Then $\Gamma_+(x) \subset K(x)$ for all $x \in M$.
\end{lemma}
\begin{proof}
The first statement is trivial.
 To prove $\Gamma_+ \subset K$, it suffices to prove that the limit $y=\lim_{j\to\infty}y_j$, $y_j \in \Gamma_{g_j}(x_j)$ belongs to $K(x)$ if $x_j \to x$.
 Since $y_j = g_j(x_j)$, by definition of relaxed limits $\limsupstar$ and $\liminfstar$ it is easy to see that
\[
	\liminfstar_{j\to\infty}g_j(x) \leq y \leq \limsupstar_{j\to\infty} g_j(x).
\]
Thus $\Gamma_+(x) \subset K(x)$.
\end{proof}

We next discuss an equivalent condition the graph convergence.
\begin{lemma} \label{SG3}
	\vspace{-1em}
\begin{enumerate}[leftmargin=2em,topsep=0em]
    \item Let $\left\{g_j\right\}^\infty_{j=1} \subset C(M)$ be a bounded sequence.
    Then the semi-limit $\Gamma_-=\liminf_{j\to\infty}\Gamma_{g_j}$ belongs to $\mathcal{B}$.
    \item Assume that $M$ is locally arcwise connected.
    If $\Gamma_-(x)$ contains both semi-limits $\liminfstar g_j(x)$ and $\limsupstar g_j(x)$, then $K(x)\subset\Gamma_-(x)$.
    Moreover, $K(x) = \Gamma_+(x) = \Gamma_-(x)$ for all $x \in M$ and $\Gamma_{g_j}$ converges to $\Gamma$ in the graph sense.
    Conversely, if $\Gamma_{g_j}$ converges to $\Gamma$ in the graph sense, then $\Gamma=\Gamma_+ = \Gamma_- = K$.
\end{enumerate}
\end{lemma}
\begin{proof} 
(1) follows from the definition and we focus on the proof of (2).
 If $\Gamma_-(x)$ contains $\partial \left(K(x)\right)$, then there is $x_j \in M$, $y_j = g_j(x_j)$ such that $x_j \to x$, $y_j \to \hat{y}$ for $\hat{y}=\liminfstar g_j(x)$ and that there exists $\overline{x}_j \in M$, $\overline{y}_j \in g_j(\overline{x}_j)$ such that $\overline{x}_j \to x$, $\overline{y}_j \to \overline{y}$ for $\overline{y}=\limsupstar g_j(x)$.

By assumption, for any $\delta>0$ there exists an arc $\gamma_j$ connecting $x_j$ to $\overline{x}_j$, lying in a $\delta$-neighborhood $B_\delta$ of $x$ provided that {$j$} is sufficiently large.
 Since $g_j$ is continuous on $\gamma_j \subset B_\delta$, the intermediate value theorem implies that $\left[ y_j,\overline{y}_j \right] \subset g_j(B_\delta)$.
 Thus $K(x) \subset \Gamma_-(x)$.
 
By Lemma \ref{SG2}, we know $\Gamma_+(x) \subset K(x)$.
By definition of $\Gamma_-$ we see $\Gamma_- \subset \Gamma_+$.
Thus $\Gamma=\Gamma_+ = \Gamma_- = K$.
The converse statement is easy to check.
The proof is now complete.
\end{proof}
We next consider an important subclass of $\mathcal{B}$.
 Let $\mathcal{A}$ be the family of $\Gamma\in\mathcal{B}$ satisfying that $\Gamma(x)$ is a closed interval for all $x\in M$.
 Let $\mathcal{A}_0$ be the subfamily of $\mathcal{A}$ such that $\Gamma(x)$ is a singleton $\{1\}$ except countably many exceptions of $x \in M$.
 Such $\Gamma$ is uniquely determined by $\{x_i\}^\infty_{i=1}$ where $\Gamma(x_i)=\left[\xi^-_i,\xi^+_i \right]$ with $\xi^-_i < \xi^+_i$ containing $1$ and $\Gamma(x)=\{1\}$ if $x \notin \{x_i\}^\infty_{i=1}$.
We call such a point $x_i$ \textit{an exeptional point} of $\Xi \in \mathcal{A}_0$,
so that $\Sigma$ is the set of all exceptional points of $\Xi$.

We next study compactness in the graph convergence.
\begin{lemma}\label{SC}
    Let $\{g_j\}_{j=1}^{\infty} \subset C(M)$ be a bounded sequence.
    Assume that
    \[
        \begin{array}{ll}
            \eta^{-}(x) < \eta^{+}(x) & \mbox{for $x \in S$,} \\
            \eta^{-}(x) = \eta^{+}(x) = 1 & \mbox{for $x \in M \setminus S$,}
        \end{array}
    \]
    where $S$ is a countable set and
    \[
        \eta^{-}(x) = \liminfstar_{j \to \infty} g_j(x),\ 
        \eta^{+}(x) = \limsupstar_{j \to \infty} g_j(x).
    \]
    If $1 \in [\eta^{-}(x),\eta^{+}(x)]$,
    then there is a subsequence $\{g_{j_k}\}$ such that
    $\Gamma_{g_{j_k}}$ converges to some $\Gamma_0 \in \mathcal{A}_0$ in the graph sense.
\end{lemma}
\begin{proof}
    We write $S = \{x_i\}_{i=1}^\infty$.
    By definition, there is a subsequence $\{g_{1,j}^{-}\}$ of $\{g_{j}\}$ such that
    \[
        \eta^{-}(x_1) = \lim_{j \to \infty} g_{1,j}^{-}(y_{1,j})
    \]
    with some $\{y_{1,j}\}$ converging to $x_1$.
    We set
    \[
        \eta_1^{+}(x_1) := \limsupstar_{j \to \infty} g_{1,j}^{-}(x_1) \le \eta^{+}(x_1).
    \]
    Since $\eta^- = \eta^+ = 1$ outside $S$,
    we see $\eta^{+}(x_1) \ge 1$.
    We take a further subsequence $\{g_{1,j}\}$ of $\{g_{1,j}^{-}\}$ so that
    \[
        \eta_1^{+}(x_1) = \lim_{j \to \infty} g_{1,j}(z_{1,j})
    \]
    with some $\{z_{1,j}\}$ converging to $x_1$.
    We repeat this procedure for $x_2,x_3,\dots$
    and find a subsequence $\{g_{\ell,j}\}_{j=1}^\infty$ so that
    \[
        \begin{aligned}
            \lim_{j \to \infty} g_{\ell,j}(y_{\ell,j}) &= \liminfstar_{j \to \infty} g_{\ell,j}(x_\ell) \le 1 \\
            \lim_{j \to \infty} g_{\ell,j}(z_{\ell,j}) &= \limsupstar_{j \to \infty} g_{\ell,j}(x_\ell) \ge 1
        \end{aligned}
    \]
    with some $\{y_{\ell,j}\}$, $\{z_{\ell,j}\}$ converging to $x_\ell$ for $\ell=1,2,\dots,k$.
    By diagonal argument, we see that $\{g_{k,k}\}_{k=1}^{\infty}$ has the property that
    \[
        \xi^-(x) := \liminfstar_{k \to \infty} g_{k,k}(x),\ 
        \xi+-(x) := \limsupstar_{k \to \infty} g_{k,k}(x)
    \]
    belong to $\Gamma_{-}(x) = \liminf_{k\to \infty} \Gamma_{g_{k,k}}$ for $x \in M$.
    We now apply Lemma~\ref{SG3}(2) to conclude that
    $\Gamma_{g_{k,k}}$ converges to $\Gamma$ with
    \[
        \Gamma(x) = [\xi^-(x),\xi^{+}(x)],\ x \in M.
    \]
    By construction, $\Gamma(x) = \{1\}$ for $x \in M \setminus S$ and $\xi^-(x) \le 1 \le \xi^+(x)$ for $x \in S$.
    Thus, $\Gamma \in \mathcal{A}_0$ so the proof is now complete.
\end{proof}
We now define several functionals when $M=\overline{I}$ or $\mathbf{T}=\mathbf{R}/\mathbf{Z}$, where $I$ is a bounded open interval in $\mathbf{R}$ and $\overline{I}=\cl I$.
 For a real-valued function $v$ on $M$ and $\eps>0$, a single-well Modica--Mortola functional is defined by
\[
	\ESMM{\eps}(v) := \frac{\eps}{2}\int_M \abs{\odfrac{v}{x}}^2 \id x
	+ \frac{1}{2\eps} \int_M F(v) \id x.
\]
Here the potential energy $F$ is a single-well potential.
 We shall assume that
\begin{enumerate}
\item[(F1)] $F \in C(\mathbf{R})$ is nonnegative and $F(v)=0$ if and only if $v=1$;
\item[(F2)] $\liminf_{|v|\to\infty} F(v) > 0$;
\item[(F2')] (growth condition) there are positive constants $c_0, c_1$ such that
\[
	F(v) \geq c_0|v|^2 - c_1 \quad\text{for all}\quad v \in \mathbf{R}.
\]
\end{enumerate}
\begin{remark}
Obviously, (F2') implies (F2).
\end{remark}
We are interested in a Gamma limit of $\ESMM{\eps}$ not in usual $L^1$-convergence but the graph convergence which is of course finer than $L^1$ topology.
 As usual, we set
\[
	G(v) = \abs{\int^v_1 \sqrt{F(\tau)} d\tau}.
\]
A typical example of $F(v)$ is $F(v)=(v-1)^2$.
 In this case,
\[
	G(v) = (v-1)^2/2.
\]
To write the limit energy for $\Xi\in\mathcal{A}_0$, let $\Sigma=\{x_i\}^\infty_{i=1}$ denote the totality of points where $\Xi(x_i)$ is a nontrivial closed interval $\left[\xi^-_i,\xi^+_i \right]$ such that $\xi^-_i \leq 1 \leq \xi^+_i$.
This set can be a finite set.
 By definition, $\Xi(x)=\{1\}$ if $x \notin \{x_i\}^\infty_{i=1}$.
 In the case that $M=\mathbf{T}$, we define
\begin{eqnarray*}
	\ESMM{0}(\Xi,\mathbf{T}) := 
	\begin{cases}
	    	\displaystyle 2\sum^\infty_{i=1} \left\{ G(\xi^-_i) + G(\xi^+_i) \right\} &\text{for }\Xi\in\mathcal{A}_0, \\
\infty &\text{otherwise.}
	\end{cases}
\end{eqnarray*}
In the case that $M=\overline{I}$, one has to modify the value when $x_i$ is the end point of $\overline{I}$.
The energy is defined by
\begin{eqnarray*}
	\ESMM{0}(\Xi,\overline{I}) := 
	\begin{cases}
	    \displaystyle \sum^\infty_{i=1} \left\{ 2(G(\xi^-_i) + G(\xi^+_i)) - \kappa_i \max (G(\xi^-_i),G(\xi^+_i))\right\} &\text{for }\Xi\in\mathcal{A}_0, \\
\infty &\text{otherwise,}
	\end{cases}
\end{eqnarray*}
where $\kappa_i=0$ if $x_i \in I$ and $\kappa_i=1$ if $x_i \in \partial I$.
 To shorten the notation, {we simply write $v_j \graphto \Xi$ by the abuse of notation if $v_j \in C(M)$ is a sequence such that $v_j\to\Xi$ ($j\to\infty$) in the sense of the graph convergence.}
 We also use $v_\eps\graphto\Xi$ as $\eps \to 0$ if $\eps$ is a continuous parameter.

We shall state that the Gamma limit of $\ESMM{\eps}$ is $\ESMM{0}$ as $\eps \to 0$ under the graph convergence.
 For later applications, it is convenient to consider a slightly general functional of form $\ESMMb{\eps}{b}(v) := \ESMM{\eps}(v)+b\left(v(a)\right)^2$, where $a\in\ringM=\operatorname{int} M$ and $b \geq 0$.
 The corresponding limit functional is 
\[
	\ESMMb{0}{b}(\Xi,M) := \ESMM{0}(\Xi,M) + b \left( \min\Xi(a) \right)^2
\] 
\begin{theorem}[Gamma limit under graph convergence] \label{G1} 
    {Assume the following conditions:
    \begin{itemize}
        \item $M=\overline{I}$ or $\mathbf{T}$;
        \item $F$ satisfies (F1) and (F2);
        \item $a \in \ringM = \operatorname{int} M$ and $b \geq 0$.
    \end{itemize}
    Then the following inequalities hold:}
\begin{enumerate}
\item[(i)] (liminf inequality) Let $\{v_\eps\}_{0<\eps<1}$ be in $H^1(M) \subset C(M)$.
 If $v_\eps\graphto\Xi\in\mathcal{B}$, then
\[
	\ESMMb{0}{b}(\Xi,M) \leq \liminf_{\eps \to 0} \ESMMb{\eps}{b}(v_\eps) 
\]
In particular, $\Xi\in\mathcal{A}_0$
\item[(i\!i)] (limsup inequality) For any $\Xi\in\mathcal{A}_0$, there is $\{w_\eps\}_{0<\eps<1} \subset H^1(M) \subset C(M)$ such that $w_\eps \graphto\Xi$ and 
\[
	\ESMMb{0}{b}(\Xi,M) = \lim_{\eps \to 0} \ESMMb{\eps}{b}(w_\eps) 
\]
\end{enumerate}
\end{theorem}
We also have a compactness result.
\begin{theorem}[Compactness] \label{C1} 
Assume that $M=\overline{I}$ or $\mathbf{T}$.
 Assume that $F$ satisfies (F1) and (F2').
 Let $\{v_{\eps_j}\}^\infty_{j=1}$ be in $H^1(M) \subset C(M)$. 
 Assume that
\[
	\sup_j \ESMM{\eps_j}(v_{\eps_j}) < \infty
\]
for $\eps_j \to 0$ as $j\to\infty$.
 Then there exists a subsequence $\left\{v_{\eps'_k}\right\}$ such that $v_{\eps'_k}\graphto\Xi$ with some $\Xi\in\mathcal{A}_0$.
\end{theorem}

By combining the Gamma convergence result and the compactness, a general theory yields the convergence of a minimizer of $\ESMMb{\eps}{b}$; see~\cite[Theorem 1.21]{Br} for example.
 Note that in the case of $b=0$, the minimum of $\ESMM{\eps}(v)$ is zero and is attained only at constant function $v=1$ so the convergence of minimizers is trivial. 
\begin{corollary} \label{CM1} 
Assume the same hypotheses of Theorem \ref{G1} and (F2').
 Let $v_\eps$ be a minimizer of $\ESMMb{\eps}{b}$ on $H^1(M)$.
 Then there is a subsequence $\left\{v_{\eps_k}\right\}^\infty_{k=1}$ such that $v_{\eps_k}\graphto\Xi_0$ with some $\Xi_0 \in \mathcal{A}_0$.
 Moreover, $\Xi_0$ is a minimizer of $\ESMMb{0}{b}$.
 Furthermore, $\Xi_0(x) = \{1\}$ if $x\neq a$ and $\Xi_0(a)=[p_0,1]$,
 where $p_0$ is a minimizer of $2 G(p)+b p^2$ with $p\in[0,1]$.
\end{corollary}
\begin{remark} \label{VM} 
If $F'(v)(v-1) \geq 0$, then $G$ is convex so that $2G(p)+b p^2$ is strictly convex for $b>0$.
 In this case, the minimizer is unique.
 If $F(v)=(v-1)^2$ so that $G(v)=(v-1)^2/2$, then $2G(p)+b p^2=(p-1)^2+b p^2$ and its minimizer is $1/(b+1)$ and its minimal value is $\ESMMb{0}{b}(\Xi_0,M)=b/(b+1)$.
\end{remark}

Our theory has an application to the Kobayashi--Warren--Carter energy~\cite{KWC98,KWC00,KWC} which can be interpreted as an Ambrosio--Tortorelli inhomogenization of the total variation energy.
 Its typical form is 
\[
	\EKWC{\eps}(u,v) := \sigma \int_{\ringM} v^2 \left|\odfrac{u}{x}\right| + \ESMM{\eps}(v)
\]
for $\sigma\geq 0$.
 The first integral denotes the total variation of $u$ with weight $v^2$.
 {See Section~\ref{sec:5} for more rigorous definition.}
 Note that if $u_x=0$ outside $a$ and $u$ jumps at $a$ with jump $1$, then
\[
	\EKWC{\eps}(u,v)=\ESMMb{\eps}{\sigma}(v)
\]
so our $\ESMMb{\eps}{\sigma}(v)$ is considered a special value of $\EKWC{\eps}(u,v)$ by fixing such $u$.
For $\Xi \in \mathcal{A}_0$, let $\Sigma = \{x_i\}^\infty_{i=1}$ be the set of all exceptional points of $\Xi$.
(Note that the set $\Sigma$ can be finite.)
 Let  $\xi^-_i = \min \Xi(x_i)$ for $x_i\in\Sigma$.
 For $u\in BV(\ringM)$, let $J_u$ denote the set of jump discontinuities of $u$, i.e.,
\[
	J_u {:=} \left\{ x \in \ringM \Bigm|
	d(x) = \bigl| u(x+0) - u(x-0) \bigr| >0 \right\}
\] 
where $u(x+0)$ (resp.\ $u(x-0)$) denotes the trace from right (resp.\ left).
 For $(u, \Xi) \in L^1(\ringM) \times \mathcal{B}$, we set
\begin{equation*}
	\EKWC{0}(u,\Xi,M) := \left \{
\begin{array}{l}
	\displaystyle \sigma \int_{\ringM\setminus (J_u \cap \Sigma)} \abs{\odfrac{u}{x}} 
	+ \sigma \sum^\infty_{i=1} d_i \left(\xi^-_i\right)^2 + \ESMM{0}(\Xi,M) \\
	\phantom{\displaystyle \sigma \int_{\ringM\setminus J_u} \left|\frac{du}{dx}\right| 
	+ \sigma \sum^\infty_{i=1} d_i \left(\xi^-_i\right)^2} \text{for}\ u\in BV(\ringM)\ \text{and}\ \Xi \in \mathcal{A}_0, \\
	\infty \quad \text{otherwise,} 
\end{array}
\right.
\end{equation*}
where $d_i = d(x_i)$.
Here $\displaystyle\int_{\Omega} \abs{\odfrac{u}{x}}$ denotes the total variation in $\Omega \subset \ringM$.
Since the measure $\abs{u_x}$ is a continuous measure outside $J_u$ so that $\abs{u_x} (\Sigma \backslash J_u) = 0$, one may replace $\Sigma \cap J_u$ by $\Sigma$ in the domain of integration in the definition of $\EKWC{0}$.
\begin{theorem}[Gamma limit] \label{G2} 
Assume that the same hypotheses of Theorem \ref{G1} concerning $M$ and $F$.
\begin{enumerate}[leftmargin=2em,topsep=0em]
\item[(i)] (liminf inequality) Let $\{v_\eps\}_{0<\eps<1}$ be in $H^1(M) \subset C(M)$.
 Assume that $v_\eps\graphto\Xi\in\mathcal{B}$ as $\eps \to 0$.
 Let $\{u_\eps\}\subset L^1(M)$ satisfy $u_\eps \to u$ in $L^1(\ringM)$ as $\eps \to 0$.
 Then
\[
	\EKWC{0}(u,\Xi,M) \leq \liminf_{\eps\to 0} \EKWC{\eps}(u_\eps,v_\eps).
\]
\item[(ii)] (limsup inequality) For any $\Xi\in\mathcal{A}_0$ and $u\in BV(\ringM)$, there exists $\{w_\eps\}_{0<\eps<1}\subset H^1(\ringM)\subset C(M)$ and $\{u_\eps\}_{0<\eps<1}\subset L^1(M)$ such that $w_\eps \graphto\Xi$ and $u_\eps\to u$ in $L^1$ satisfying 
\[
	\EKWC{0}(u, \Xi, M) = \lim_{\eps\to\infty} \EKWC{\eps}(u_\eps,w_\eps).
\]
\end{enumerate}
\end{theorem}
\begin{remark}
	\vspace{-1em}
\begin{enumerate}[leftmargin=2em,topsep=0em]
\item[(i)] From the proof of Lemma~\ref{SG1},
it suffices to assume $u_\eps \to u$ in $L^1_{\mathrm{loc}}(\mathring{M}\backslash\Sigma_0)$ where
\[
	\Sigma_0 = \left\{ x \in M \mid \min \Xi (x) = 0 \right\}
\]
in the statement of Theorem~\ref{G2} (i).
 Since $\Xi$ must be in $\mathcal{A}_0$ and $E^0_{\mathrm{sMM}}(\Xi,M) < \infty$, this set $\Sigma_0$ must be a finite set.
\item[(ii)] We may add a fidelty term $\lambda \|u-g\|^2_{L^2(\mathring{M})}$ to energies $E^\eps_{\mathrm{KWC}}$, $E^0_{\mathrm{KWC}}$ for $\lambda>0$ with given $g \in L^2(\mathring{M})$ like the Ambrosio-Tortorelli functional $\mathcal{E}^\eps$ and the Munford-Shah functional.
 More precisely, the statement of Theorem~\ref{G2} is still valid for 
\begin{align*}
 	& \EKWC{\eps,\lambda}(u,v) 
 	:= \EKWC{\eps}(u,v) + \lambda \int_{\mathring{M}} \abs{u-g}^2 \id x \\
 	& \EKWC{0,\lambda}(u,\Xi,M) 
 	:= \EKWC{0}(u,\Xi,M) + \lambda \int_{\mathring{M}} \abs{u-g}^2 \id x.
\end{align*}
\end{enumerate}
\end{remark}

The next compactness result easily follows from the compactness (Theorem~\ref{C1}) in $\mathcal{B}$ and $L^1$-compactness of $BV(\Omega)$, where $\Omega$ is an open set such that $\overline{\Omega} \subset \mathring{M}\backslash \Sigma_0$.
\begin{theorem}[Compactness]
Assume the same hypothesis of Theorem~\ref{C1} concerning $M$ and $F$.
 Let $\lambda >0$ be fixed.
 Let $\left\{v_{\eps_j}\right\}^\infty_{j=1}$ be in $H^1(M) \subset C(M)$ and  $\left\{u_{\eps_j}\right\} \subset L^2(\mathring{M})$.
 Assume that
\[
	\sup_j \EKWC{\eps_j,\lambda} \left(u_{\eps_j},v_{\eps_j}\right)<\infty
\]
for $\eps_j \to 0$.
 Then there exists a subsequence $\left\{\left(u_{\eps'_k},v_{\eps'_k}\right)\right\}$ such that $u_{\eps'_k} \to u$ in $L^1_{\mathrm{loc}}(\mathring{M}\backslash \Sigma_0)$ with some $u \in L^1_{\mathrm{loc}}(\mathring{M}\backslash \Sigma_0)$ and that $v_{\eps'_k} {\xrightarrow{g}} \Xi$ with some $\Xi \in \mathcal{A}_0$.
 Here
\[
	\Sigma_0 = \left\{ x \in M \mid \min \Xi (x) = 0 \right\}.
\]
\end{theorem}
By combining the Gamma convergence result and the compactness, a general theory yields the convergence of a minimizer of $\EKWC{\eps,\lambda}$; see~\cite[Theorem 1.21]{Br} for example.
\begin{corollary}
Assume the same hypothesis of Theorem~\ref{G1}.
 Let $(u_\eps, v_\eps)$ be a minimizer of $\EKWC{\eps,\lambda}$.
 Then, there is a subsequence $\left\{\left(u_{\eps_k},v_{\eps_k}\right)\right\}^\infty_{k=1}$ such that $v_{\eps_k} {\xrightarrow{g}} \Xi$, $u_{\eps_k} \to u$ in $L^1_{\mathrm{\mathrm{loc}}}(\mathring{M}\setminus \Sigma_0)$
 and that the limit $(u,\Xi_0)$ be a minimizer of  $\EKWC{\eps,\lambda}$.
 Here $\Sigma_0 = \left\{ x \in M \mid \min \Xi_0 (x) = 0 \right\}$.
\end{corollary}

\section{Unfolding by arc-length parameters} \label{sec:3} 

For a bounded open interval $I$ let $u$ be a real-valued $C^1$ function on $\overline{I}$, that is, $u \in C^1(\overline{I})$.
 To simplify notation, we set $I=(0,r)$.
 Then the arc-length parameter $s$ of the graph curve $y=u(x)$ is defined as
\[
	s(x) = s_u(x) := \int^x_0 \left( 1+u^2_x(z) \right)^{1/2} \id z.
\]
One is able to extend this definition for general $u \in BV(I)$.
 By definition, $s(\cdot)$ is strictly monotone increasing.
 It is easy to see that {$s(\cdot)$} is continuous
 if and only if the derivative $u_x$ has no point mass, {that is, $u$ has no jump, which is equivalent to $u \in C(\overline{I})$.}
 The inverse function $x=x(s)$ of $s=s(x)$ is always Lipschitz with Lipschitz constant 1,
 that is, $\operatorname{Lip}(x) \leq 1$.
 Indeed, since $\displaystyle\odfrac{x}{s}=(1+u^2_x)^{-1/2}$, the inequality $\displaystyle\abs{\odfrac{x}{s}}\leq 1$ always holds.
 For $u \in C(\overline{I}) \cap BV(I)$, we define an unfolding $U$ by arc-length parameter of the form
\[
	U(s) = u \left( x(s) \right).
\]
 The function $U$ is defined on $\overline{J}$ with $J=(0,L)$,
 where $L$ is the length of the graph $u$ on $\overline{I}$.

We begin with several basic properties of the unfoldings.
\begin{lemma} \label{L1} 
Assume that $u \in W^{1,1}(\overline{I})$.
\begin{enumerate}
\item[(i)]  
 $U$ is Lipschitz continuous on $J$.
 More precisely, $\mathrm{Lip}(U) \leq 1$.
\item[(ii)]
 The total variation of $U$ on $J$ equals that of $u$ in $I$, that is,
\[
	\mathrm{TV}(u) = \mathrm{TV}(U).
\]
\end{enumerate}
\end{lemma}
\begin{proof}
\begin{enumerate}
\item[(i)]  
Since
\[
	U_s = \frac{u_x}{\left( 1+u^2_x \right)^{1/2}},
\]
$\mathrm{Lip}(U) \leq 1$ is rather clear.
\item[(ii)]
By definition,
\[
	\mathrm{TV}(U) = \int^L_0\abs{U_s}{\id} s = \int^r_0\abs{u_x}\id x = \mathrm{TV}(u).
\]
\end{enumerate}
\end{proof}
Since $q = p/(1+p^2)^{1/2}$ is equivalent to $p=q/(1-q^2)^{1/2}$, we see that
\[
	\odfrac{x}{s} = \frac{1}{(1+u^2_x)^{1/2}} = \left( 1-U^2_s \right)^{1/2}
\]
We next discuss compactness for unfoldings and the lower semicontinuity of $\mathrm{TV}(\cdot)$.
\begin{lemma} \label{L2} 
Assume that $\{u^\eps\}_{0<\eps<1} \subset W^{1,1}(I)$ with a bound for $\mathrm{TV}(u^\eps)$ and $\|u^\eps\|_\infty$.
 Then there is a subsequence such that $U^\eps$ tends to some function $V$ with $\mathrm{Lip}(V) \leq 1$ uniformly in a domain of definition of $V$.
 Moreover, $\mathrm{TV}(V) \leq \liminf_{\eps \to 0} \mathrm{TV}(u^\eps)$.
\end{lemma}
\begin{proof}
Since $\mathrm{TV}(u^\eps)$ is bounded, so is the length $L_\eps$ of the graph of $u^\eps$.
 The existence of convergent subsequence follows from the Ascoli-Arzela theorem.
 A basic lower semicontinuity of $\mathrm{TV}(\cdot)$ yields
\[
	\mathrm{TV}(V) \leq \liminf_{\eps \to 0} \mathrm{TV}(U^\eps).
\]
The right-hand side equals $\mathrm{TV}(u^\eps)$ as proved in Lemma \ref{L1} (i\!i) so the proof is now complete.
\end{proof}
{We raise a question whether or not a Lipschitz function $V$ on $J$ with $\mathrm{Lip}(V) \leq 1$ can be written as $u\left(x(s)\right)$.
This is in general not true if there is a non trivial interval such that $U_s=1$ (or $U_s = -1$). Indeed, if $U_s = \pm 1$, then $x(s)$ is not invertible.

In spite of this lack of the correspondence, however, the following lemma states that the limit of the unfolding contains the information on the pointwise behaviour of $u^{\eps_k}$. (See also Figure~\ref{fig:T3}.)}
\begin{theorem} \label{T3} 
Assume that $\{u^{\eps_k}\}^\infty_{k=1}\subset W^{1,1}(I)$ with a bound for $\mathrm{TV}(u^{\eps_k})$
{and its unfolding} $U^{\eps_k}$ converges uniformly to $V$ in a domain $J$ of definition of $V$.
 (The domain $J$ must be a bounded interval by a bound of $\mathrm{TV}(u^\eps)$.)
 {If $x^{\eps_k}$, the inverse of the arc-length parameter of $u^{\eps_k}$, converges uniformly to a limit $\overline{x}$ in $J$,
 then}
\[
\begin{aligned}
\left( \limsupstar_{k\to\infty} \,u^{\eps_k} \right)(x) &= \max\left\{ V(s) \mid \overline{x}(s)=x \right\} \\
\left( \liminfstar_{k\to\infty} \,u^{\eps_k} \right)(x) &= \min\left\{ V(s) \mid \overline{x}(s)=x \right\}, \quad x \in \overline{I}.
\end{aligned}
\]
\end{theorem}
\begin{figure}
    \centering
    \includegraphics{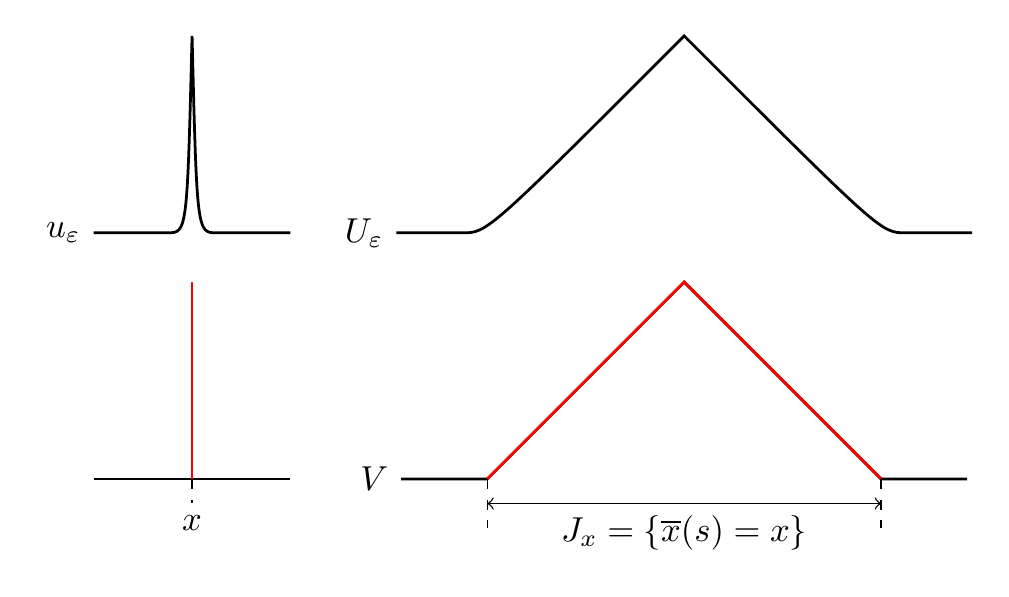}
    \caption{The visual example of Theorem~\ref{T3}.
    The sequence of function $u^\eps$ is unfolded to $U^\eps$ (the upper right figure).
    When $U^\eps$ converges uniformly to $V$,
    the corresponding limit of $u^\eps$ can be no longer captured as single-valued function
    but is possibly multi-valued.
    The red part of the graph of $V$ (the lower right image), however,
    corresponds to the multi-valued part (the red part in the lower left image)
    and its maximum and minimum coincide with the upper and lower relaxed limit of $u^\eps$, respectively.}
    \label{fig:T3}
\end{figure}
\begin{proof}
Since the proof is symmetric, we only give a proof for $\limsupstar$.
 Let $J_x = \left\{ s \in J \mid \overline{x}(s)=x \right\}$.
 We take $s_* \in J_x$ such that
\[
	V(s_*) = \max_{J_x}V.
\]
Since $V$ is the limit of $U^\eps$, we have
\[
	V(s_*) = \lim_{k\to\infty} u^{\eps_k} \left( x^{\eps_k}(s_*) \right) 
	\leq \left( \limsupstar_{k\to\infty} u^{\eps_k}  \right)(x).
\]

To prove the converse inequality, we set
\[
	J^\sigma_x = \left\{ s \in J \bigm| \left| \overline{x}(s)-x \right| \leq \sigma \right\}.
\]
Since $x^{\eps_k}$ converges to $\overline{x}$ uniformly in $J$, for sufficiently large $k$, say $k>k_0(\sigma)$,
\[
	x^{\eps_k} (J^{2\sigma}_x) \supset \left\{ y \in I \bigm| |y-x| \leq \sigma \right\};
\]
here $k_0(\sigma)$ can be taken so that $k_0(\sigma)\to\infty$ as $\sigma\to\infty$ and $k_0(\sigma)>1/\sigma$.
 We thus observe that
\[
	\sup_{|y-x|\leq\sigma} u^{\eps_k}(y) 
	\leq \sup \left\{ u^{\eps_k}(y) \bigm| y \in x^{\eps_k}(J^{2\sigma}_x) \right\}
	= \sup \left\{ U^{\eps_k}(s) \bigm| s \in J^{2\sigma}_x \right\}
\]
for $k>k_0(\sigma)$.
 Sending $\sigma\to 0$, we observe that
\[
	\lim_{\sigma\downarrow 0} \sup_{\substack{|y-x|\leq\sigma \\ k>k_0(\sigma)}} u^{\eps_k}(y)
	\leq \max_{s \in J_x}V(s).
\]
The left-hand side agrees with $\displaystyle\limsupstar_{k\to\infty} u^{\eps_k}$ since
\begin{align*}
	& \left\{ (y,k) \bigm| |y-x| < 1/k_0(\sigma),\ k>k_0(\sigma) \right\} \\
	\subset & \left\{ (y,k) \bigm| |y-x| < \sigma,\ k>k_0(\sigma) \right\} \\
	\subset & \left\{ (y,k) \bigm| |y-x| < \sigma,\ k>1/\sigma \right\}.
\end{align*}
We thus conclude that
\[
	\left( \limsupstar_{k\to\infty} u^{\eps_k} \right) (x)
	\leq \max\left\{ V(s) \bigm| s \in J_x \right\}. 
\]
The proof is now complete.
\end{proof}
We next prove the inequality connecting the total variation and the relaxed limit in terms of the unfolding. (see Figure~\ref{fig:thm5}.)
\begin{figure}[tbp]
    \centering
    \begin{tikzpicture}[scale=0.75]
    \draw[thick] (0,0) node[left] {$0$} -- (5,0);
    \draw[thick] (1,0) node[below] {$x_1$} -- (1,2);
    \draw[thick] (3,0) node[below] {$x_2$} -- (3,1);
    \draw[thick] (4,0) node[below] {$x_3$} -- (4,0.5);
    \begin{scope}[yshift=-4cm]
        \draw[thick] (0,0) node[left] {$0$} -- (1,0) -- (3,2) -- (5,0) -- (7,0) -- (8,1) -- (9,0) -- (10,1) -- (11,0) -- (12,0) -- (12.5,0.5) -- (13,0) -- (14,0);
        \node[left] at (0,2) {$V$};
        \draw[dashed] (1,0) -- (1,-1);
        \draw[dashed] (5,0) -- (5,-1);
        \draw[<->] (1,-0.5) -- node[below] {$J_{x_1,1}$} (5,-0.5);
        \draw[dashed] (7,0) -- (7,-1);
        \draw[dashed] (9,0) -- (9,-1);
        \draw[<->] (7,-0.5) -- node[below] {$J_{x_2,1}$} (9,-0.5);
        \draw[dashed] (11,0) -- (11,-1);
        \draw[<->] (9,-0.5) -- node[below] {$J_{x_2,2}$} (11,-0.5);
        \draw[dashed] (12,0) -- (12,-1);
        \draw[dashed] (13,0) -- (13,-1);
        \draw[<->] (12,-0.5) -- node[below] {$J_{x_3,1}$} (13,-0.5);
    \end{scope}
\end{tikzpicture}
    \caption{The visual explanation of Theorem~\ref{T4}. If the graph of  $\{u^{\eps_k}\}$ converges to the graph as the top,
    its unfolding converges to $V$, whose graph is like the bottom.
    Then $J_{x_i} = \{s \in J \mid \overline{x}(s) = x_i\}$ can be decomposed as the union of $\{J_{x_i,j}\}_{j=1,2,\dots}$ by labelling the disjoint intervals where $V$ does not vanish.
    }
    \label{fig:thm5}
\end{figure}
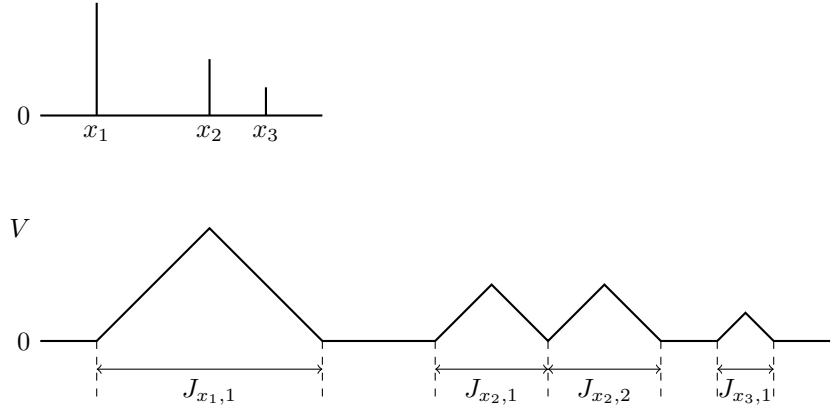
\begin{theorem} \label{T4} 
Assume the same hypothesis of Theorem \ref{T3}.
 Then the set $\Sigma$ of points $x$ where
\[
	\limsupstar_{k\to\infty} u^{\eps_k} (x) 
	> \liminfstar_{k\to\infty} u^{\eps_k} (x) 
\]
{has at most countable cardinality.}
 Assume furthermore that outside $\Sigma$ the limit must be zero and $\displaystyle\liminf_{k\to\infty}{}_* u^{\eps_k} (x)=0$ for all $x \in I$.
 Then 
\[
	\liminf_{k\to\infty} \mathrm{TV}(u^{\eps_k})
	\geq \sum_{x\in\Sigma} 2 \chi(x) \limsupstar_{k\to\infty} u^{\eps_k} (x),
\]
where $\chi(x)=1$ for $x \in I$ and $\chi(x)=1/2$ for $x\in\partial I$.
\end{theorem}
\begin{proof}
 If $\#\Sigma$ is more than countable,
 then there is an infinite number of 
intervals $J_{x_i}$ such that $\max V - \min V > c_0$ with some $c_0 > 0$.
This is impossible by Theorem~\ref{T3}, since $\mathrm{TV}(V) < \infty$.
Thus, $\Sigma$ is at most a countable set.

 We write $\Sigma=\{x_i\}^\infty_{i=1}$ and $J_i=J_{x_i}$.
 We set $\rho_i=\max_{J_i}V$.
 The cases devided into two cases whether or not $\overline{J_i}$ contains a boundary point of $J$.
 The total variation is estimated so 
\[
	\mathrm{TV}(V) \geq \sum^\infty_{i=1} 2 \chi_i \rho_i,
\]
where $\chi_i=\chi(x_i)$.
 Thus Theorem \ref{T3} and Lemma \ref{L2} yield the desired result.
\end{proof}
We decompose $J_{i}$ by
\[
	J_{i} := \cl \left(\bigcup^\infty_{j=1} {J_{x_i,j}}\right),
\]
where $V>0$ in an open interval $J_{x_i,j}$ and $V=0$ on $\partial J_{x_i,j}$.
 The union can be finite.
 
 We introduce $\overline{\chi}$ on subsets of $J_{x_i}$ which reflects behavior finer than that of $\chi$ on the boundary.
 We set for $x = x_i \in \Sigma$,
 \[
    \overline{\chi}(J_{x_i,j}) = 
    \begin{cases}
        1 & \mbox{if $J_{x_i,j}\cap \partial I = \emptyset$},\\
        1/2 &\mbox{otherwise}.
    \end{cases}
 \]
 By definition
\[
	\mathrm{TV}(V) \geq \sum^\infty_{i=1}  \sum^\infty_{j=1} 2 \overline{\chi}(J_{x_i,j})\rho_{x_i,j}
\]
with
\[
	\rho_{x_i,j} = \max_{J_{x_i,j}} V.
\]
Similarly to obtain Theorem \ref{T4}, we are able to prove a stronger result.
\begin{theorem} \label{Inff} 
Assume the same hypothesis of Theorem \ref{T3}.
 Then
\[
	\liminf_{k\to\infty} \mathrm{TV}(u^{\eps_k}) \geq 
	\sum_{x \in \Sigma} \sum^\infty_{j=1} 2 \overline{\chi}(J_{x,j})\rho_{x,j}, 
\]
where $\rho_{x,j}$ and $\overline{\chi}$ are determined from $V$ as above.
\end{theorem}
%
%
%

%
%
%
%
\section{Proof of convergence of functional and compactness} \label{sec:4} 

We shall prove the characterization of the Gamma limit of the single-well Modica--Mortola functional by the results of the previous section on unfoldings.
\begin{proof}[Proof of Theorem \ref{G1}] \ \\
(i) (liminf inequality)
We discuss the case $M=\overline{I}$.
 We may assume $I=(0,r)$.
 Assume that $v_\eps \graphto \Xi \in \mathcal{B}$ with $v_\eps \in H^1(M)$.
 By the Modica--Mortola inequality which follows from $\alpha^2+\beta^2 \geq 2\alpha\beta$ for numbers we have
\[
	\ESMM{\eps}(v_\eps) 
	\geq \int_M \left| \odfrac{v_\eps}{x} \right| \sqrt{F(v_\eps)} \id x
	= \int_M \left| G(v_\eps)_x \right| \id x.
\]
The right-hand side equals $\mathrm{TV}(u^\eps)$ if one sets $u^\eps=G(v_\eps) {\geq 0}$.
 We may assume that $\ESMM{\eps}(v_\eps)$ is bounded for $\eps \in (0,1)$
 so that $\mathrm{TV}(u^\eps)$ is bounded for $\eps \in (0,1)$
 and that $\int_M F(v_\eps)\id x\to 0$ as $\eps \to 0$.
 By (F2), the latter convergence implies that $v_\eps\to 1$ in measure.
 By taking a subsequence,
 we see that $v_{\eps'}\to 1$ a.e.
 so that $u^{\eps'}\to 0$ a.e.
 This implies that
\[
	\liminfstar_{\eps \to 0} u^\eps(x) = 0
	\quad\text{for all}\quad x \in M.
\]
By taking a subsequence,
we may assume that the inverse function $x^{\eps_k}$ of the arc-length parameter of $u^{\eps_k}$ converges to some $\overline{x}$.
 Applying Theorem~\ref{Inff},
 we see that 
\[
	\liminf_{k \to \infty} \ESMM{\eps_k}(v_{\eps_k})
	\geq \sum_{x\in\Sigma} \sum^\infty_{j=1} 2 \overline{\chi}(J_{x,j})\rho_{x,j},
\]
where $\Sigma$ is the set where $\limsupstar u^{\eps_k} (x)>0$ and $\rho_{x,j}$ is determined by limit $V$ of $u^\eps$.
Note that $\Sigma$ is at most countable.
If $v_\eps \xrightarrow{g} \Xi$,
then $u^\eps 
\xrightarrow{g} \Theta$ and $\Theta (x) = \{0\}$ if $x \notin \Sigma$ and
\[
     \Theta (x_i) = \left[ 0, \max \left(G (\xi^+_i),G (\xi^-_i) \right) 
\right]
     \quad \text{for} \quad x_i \in \Sigma
\]
by Lemma~\ref{SG3}.
By Theorem~\ref{T3}, at least one of $\rho_{x_i,j}$ should be 
equal to $\max\left(G (\xi^+_i),G (\xi^-_i) \right)$.
However, if $\xi^-_i < 1 < \xi^+_i$, then $v_\eps-1$ is 
sign-changing near $x_i$.
In this case, one of $\rho_{x_i,j}$'s must be equal to $\min \left(G 
(\xi^+_i),G (\xi^-_i) \right)$.
Thus we observe that
\[
	\sum^\infty_{j=1} \overline{\chi}(J_{x_i,j})\rho_{x_i,j} \geq
	G(\xi^+_i) + G(\xi^-_i), \quad
	x_i \in \Sigma \cap I.
\]
If $x \in \Sigma \cap \partial I$,
one has to be more careful.
For $x_i \in \Sigma \cap \partial I$, we see that
\begin{equation}\label{eq:bdry}
        \sum^\infty_{j=1} \overline{\chi}(J_{x_i,j})\rho_{x,j} \geq
	\min (G(\xi^+_i),G(\xi^-_i)) + \frac{1}{2}\max(G(\xi^+_i),G(\xi^-_i)).
\end{equation}
Indeed, without loss of generality, we assume that $G(\xi_i^{-}) < G(\xi_i^{+})$.
When $G(\xi_i^-) = 0$,
\eqref{eq:bdry} is rather easy to prove since the right hand side is equal to $\dfrac{1}{2}G(\xi_i^+)$, and then we may assume that $G(\xi_i^-)>0$.
Then there are at least two indices denoted by $j =1,2$, such that
\[
    \overline{\chi}(J_{x_i,1}) = \frac{1}{2},\ 
    \overline{\chi}(J_{x_i,2}) = 1,\ 
    \mbox{and}\ \{\rho_{x_i,1},\rho_{x_i,2}\} = \{G(\xi_i^-),G(\xi_i^+)\}.
\]
The left hand side is dominated from below by
\[
  \overline{\chi}(J_{x_i,1})\rho_{x_i,1} + \overline{\chi}(J_{x_i,2})\rho_{x_i,2}
  = \frac{1}{2}\rho_{x_i,1} + \rho_{x_i,2}.
\]
The right hand side is minimized in the case that $\rho_{x_i,1} = G(\xi_i^+)$ and $\rho_{x_i,2} = G(\xi_i^-)$.
We thus obtain the inequality~\eqref{eq:bdry}.

we now conclude that
\[
\begin{aligned}
    \liminf_{k \to \infty} \ESMM{\eps_k}(v_{\eps_k}) &\geq \sum_{i=1}^{\infty} 2(1-\kappa_i)(G(\xi^+_i)+G(\xi^-_i)) \\
    &\phantom{\geq} +\kappa_i \left[
        2 \min (G(\xi^+_i),G(\xi^-_i)) 
        + \max (G(\xi^+_i),G(\xi^-_i)) 
    \right] \\
    &= \ESMM{0}(\Xi,\overline{I}),
\end{aligned}
\]
{which is} the desired liminf inequality for $b=0$.
 Since $v_\eps \graphto \Xi$,
 we see that
\[
	\liminf b \left(v_\eps(a)\right)^2 \geq b \left(\min \Xi(a)\right)^2.
\]
Thus the desired liminf inequality follows for $b>0$.
 The case $M=\mathbf{T}$ is easier since there is no boundary point.\\[1em]
(ii) (limsup inequality)
 This follows from explicit construction of function $w_\eps$ as for the standard double-well Modica--Mortola functional.
 For $\xi \neq 1$ ($\xi\in\mathbf{R}$) and $x>0$, let $v(x,\xi)$ be a function determined by
\[
	\left| \int^v_\xi \left( \frac{1}{\sqrt{F(\rho)}} \right) d\rho \right| = x.
\]
This equation is uniquely solvable by (F1) for all $x \in [0,x_*)$ with
\[
	x_* := \left| \int^1_\xi \left( \frac{1}{\sqrt{F(\rho)}} \right) d\rho \right|.
\]
Note that $v$ solves the initial value problem
\begin{eqnarray*}
\left\{
\begin{array}{l}
	\displaystyle\odfrac{v}{x} = \sqrt{F(v)}, \quad x \in (0,x_*) \\[1em]
	v(0,\xi) = \xi
\end{array}
\right.
\end{eqnarray*}
although this problem may admit many solutions.
 We also note that $v$ is monotone and that
\[
	\lim_{x \to x_*} v(x,\xi) = 1
\]
including the case $x_*=\infty$.
 We consider the even extention of $v$ and still denote by $v$, that is, $v(x,\xi)=v(-x,\xi)$ for $x \in (-x_*,0]$.
 We next translate and rescale $v$.
 Let $v_\eps$ be of the form
\[
	v_\eps(x,z,\xi) := v \left( \frac{x-z}{\eps}, \xi \right),
	\quad x \in \mathbf{R}.
\]
By the equality case of the Modica--Mortola functional, we see that
\[
	\ESMM{\eps} (v_\eps) = \int_M \left|\odfrac{v_\eps}{x}\right| \sqrt{F(v_\eps)} \id x
	=  \int_M \left| G(v_\eps)_x \right| \id x.
\]
The right-hand side is estimated from above by
\[
	2 \left( G(\xi)- G(1) \right) = 2G(\xi)
\]
and if $z$ is a boundary point of $M$, we may replace $2G(\xi)$ by $G(\xi)$.

In order to explain the the main idea of the proof, we first study the case when all $\xi_i^+=1$
although logically we need not distinguish this case from general case.
If all $\xi^+_i = 1$, then it is easy to construct the desired $w_\eps$ by setting
\[
	w_\eps(x) = \min_{x_i\in\Sigma} v(x, x_i, \xi^-_i).
\]
Indeed, we still have
\[
	\ESMM{\eps} (w_\eps) = \int_M \left|G(w_\eps)_x\right| \id x
\]
and evidently this total variation is dominated from above by
\[
	\sum^\infty_{i=1} 2 \chi_i G (\xi^-_i).
\]
(The first identity can be proved by approximating $w_\eps$ by 
minimum of finitely many $w_\eps$'s.)
We thus observe that $\ESMM{\eps} (w_\eps)  \leq \ESMM{0}(\Xi, \mathbf{T})$ for all $\eps>0$.
 The graph {convergence} $w_\eps \xrightarrow{g} \Xi$ is rather clear since
\[
	w_\eps(x,0, \xi) \xrightarrow{g} \Xi_0
\]
on any bounded closed interval as $\eps \to 0$, where
\begin{equation*}
	\Xi_0(x) = \left \{
	\begin{array}{ll}
	1 & x \neq 0 \\
	\left[\xi,1\right] & x=0.
\end{array}
	\right.
\end{equation*}

The proof for general $\xi^\pm_i$ is more involved.
 For $\delta > 0$, we cut off $v$ by setting as follows:
 For $\xi < 1$,
\begin{equation*}
	v^\delta(x, \xi) = \left \{
	\begin{array}{ll}
	v(x, \xi) & \text{if}\quad v(x)\leq 1-\delta\beta,\ \beta=|\xi-1| \\
	\left(|x|+c\right)\wedge 1 & \text{if}\quad v(x)\geq 1-\delta\beta,
\end{array}
	\right.
\end{equation*}
and for $\xi > 1$,
\begin{equation*}
	v^\delta(x, \xi) = \left \{
	\begin{array}{ll}
	v(x, \xi) & \text{if}\ v(x)\geq 1+\delta\beta,\ \beta=|\xi-1| \\
	\left(-|x|+c'\right)\vee 1 & \text{if}\ v(x)\leq 1+\delta\beta,
\end{array}
	\right.
\end{equation*}
where constants $c,c'$ are taken so that $v^\delta$ is (Lipschitz) continuous. {(See Figure~\ref{fig:v}.)}
 We rescale and translate this $v^\delta$ and set
\[
	v^\delta_\eps(x, z, \xi) := v^\delta \left(\frac{x-z}{\eps}, \xi \right)
	\quad x \in \mathbf{R}.
\]
We consider the case when $\xi < 1$. Since $\displaystyle \odfrac{v^\delta_\eps}{x} = \frac{1}{\eps}\sqrt{F(v^\delta_\eps)}$ for $v^\delta_\eps \leq 1-\delta$, we see that for $z \in M$
\begin{equation}  \label{MME}
\begin{aligned}
	\ESMM{\eps}(v^\delta_\eps)
	& = \frac{\eps}{2} \int_M \left|  \odfrac{v^\delta_\eps}{x} \right|^2 \id x + \frac{1}{2\eps} \int_M F(v^\delta_\eps)\id x \\
	& \leq \int_{v_\eps(x)<1-\delta} \left| G(v_\eps)_x \right| \id x  \\ 
	&\phantom{=} + 2 \left\{ \frac{\eps}{2} \left( \frac{1}{\eps} \right)^2 \cdot \delta\beta\eps + \frac{1}{2\eps} \max \Bigl\{ F(\rho) \Bigm| 1-\delta \leq \rho \leq 1 \Bigr\} \delta\beta\eps \right\} \\
	& \leq 2G(\xi) + 2\beta\delta 
\end{aligned}
\end{equation}
for sufficiently small $\delta$, say $\delta<\delta_F$, since $F(\rho)\to 0$ as $\rho\to 1$.
 This $\delta_F$ depends only on $F$.
 A similar argument for $\xi > 1$ yield the same estimate~\eqref{MME}.
 
 We first consider the case when $M=\mathbf{T}$.
 Let $\eta=\eta(\eps, \delta, \xi)$ be a number such that $\operatorname{supp} \left(v^\delta_\eps-1\right) = [z-\eta, z+\eta]$.
 For $x_i \in \Sigma$, we set
\begin{equation*}
	v^\delta_{\eps, i}(x) := \left \{
	\begin{array}{ll}
	v^\delta_\eps(x, x_i, \xi^-_i) & \text{if}\quad x \in (x_i - \eta^-_i, x_i + \eta^-_i) \\
	v^\delta_\eps(x, x_i + \eta^+_i + \eta^-_i, \xi^+_i) & \text{otherwise,}
\end{array}
	\right.
\end{equation*}
where $\eta^-_i = \eta(\eps, \delta, \eta^-_i)$ and $\eta^+_i = \eta(\eps, \delta, \eta^+_i)$. (see Figure~\ref{fig:v}.)
\begin{figure}[tbp]
    \centering
    \includegraphics{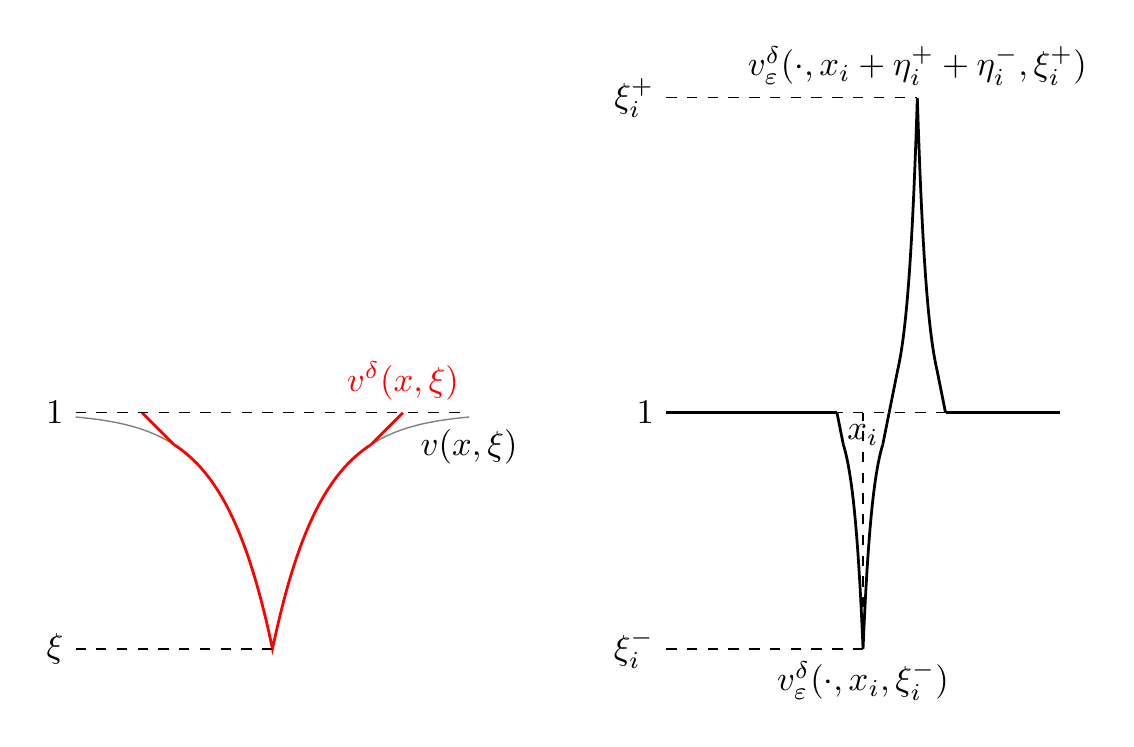}
    \caption{(left) The construction of $v^\delta(\cdot,\xi)$. In order to ensure the finiteness of the support, we take the cutoff by affine functions.
    (right) The construction of $v^\delta_{\eps,i}$ for general $\xi_{i}^{\pm}$. It is constructed by combining $v^\delta_\eps(\cdot,x_i,\xi_i^-)$ and $v^\delta_\eps(\cdot,x_i,\xi_i^+)$ with shift in order that their supports touch at their endpoints.}
    \label{fig:v}
\end{figure}
 This function is (Lipschitz) continuous and is strictly monotone from $x_i-\eta^+_i-\eta^-_i$ to $x_i$.
 For $v^\delta_{\eps,i}(x)$ by \eqref{MME}, we see that 
\begin{equation} \label{MME2}
	\ESMM{\eps} \left(v^\delta_{\eps, i}\right)
	\leq 2 \left( G(\xi^+_i) + G(\xi^-_i) \right) + 4 \beta_i \delta,
\end{equation}
where $\beta_i = \max\left( |\xi^+_i -1|, |\xi^-_i -1| \right)$.

Our goal is to construct $w_\eps$ such that $w_\eps \xrightarrow{g} \Xi$ and for each $\mu>0$, there is $\eps_\mu>0$ such that if $\eps<\eps_\mu$ then
\begin{equation} \label{eq:G1}
	\ESMM{\eps}(w_\eps) \leq \ESMM{0} (\Xi, \mathbf{T}) + \mu.
\end{equation}
We order $x_i \in \Sigma$ so that $\beta_i$ is decreasing.
{We note that $\{\beta_i\}$} must converge to zero because $\sum^\infty_{i=1} \left( G(\xi^+_i) + G(\xi^-_i) \right)<\infty$.
 For each $v^\delta_{\eps, i}$, we set $\delta = \delta_i = \delta_i(\mu)$ such that $\sum^\infty_{i=1} 4 \beta_i \delta_i < \mu$; this is, of course, possible for example by taking $\delta_i = 2^{-i-2}\mu$.
 Let $j(\mu,\eps)>0$ be the maximum number such that the support of $\left\{ v^{\delta_i}_{\eps, i}-1\right\}^{j(\mu,\eps)}_{i=1}$ is mutually disjoint.
 We set
\[
	w^\mu_\eps(x) := 1 + \sum^{j(\mu,\eps)}_{i=1} \left(v^{\delta_i(\mu)}_{\eps, i}(x) - 1\right)
\]
and observe by \eqref{MME2} that
\begin{align}  \label{eq:G2}
	\ESMM{\eps}(w^\mu_\eps)
	& \leq \sum^{j(\mu,\eps)}_{i=1} 2\left( G(\xi^+_i) + G(\xi^-_i) \right) + \sum^{j(\mu,\eps)}_{i=1} 4 \beta_i \delta_i \\ \nonumber
	& \leq \ESMM{0} (\Xi, \mathbf{T}) + \mu \quad \text{for all} \quad \eps>0.
\end{align}
Since $j(\mu,\eps)\to\infty$ as $\eps \to 0$, we see that $w^\mu_\eps \xrightarrow{g} \Xi$ as $\eps\to 0$ for each $\mu>0$.
 The desired $w_\eps$ is obtained as a kind of diagonal argument.
 Indeed, for a given $\nu>0$, we take $\eps=\eps(\nu, \mu)$ such that
\[
	d_H \left(\Gamma_{w_\eps^\mu}, \Xi \right) < \nu
\]
for $\eps \in \left(0, \eps(\nu, \mu)\right)$.
 We may assume that $\eps(\nu, \mu)$ is monotone in $\nu$ and $\mu$,
 that is,
 $\eps(\nu_2,\mu_2) \le \eps(\nu_1,\mu_1)$
 if $\nu_1 {\ge} \nu_2$ and $\mu_1 {\ge} \mu_2$.
 We then set
\[
	w_\eps := w^{\mu_\ell}_\eps \quad \text{for} \quad \eps \in \left[\eps(\nu_{\ell+1}, \mu_{\ell+1}), \eps(\nu_\ell, \mu_\ell)\right),
\] 
where $\nu_\ell,\ \mu_\ell \downarrow 0$ as $\ell\to\infty$.
 We now observe that $w_\eps \xrightarrow{g} \Xi$ and by \eqref{eq:G1} the desired estimate \eqref{eq:G1} holds for $\eps_\mu = \eps(\nu_\ell, \mu_\ell)$ for $\mu_\ell<\mu$.

We thus proved the limsup inequality for $\ESMM{\eps}$ for $M=\mathbf{T}$.
 If $b>0$, we may assume that $\xi^-_i = \min \Xi(a)<1$.
 It is easy to see that $w_\eps(a) = \xi^i_-$ for all $\eps>0$ by construction.
 Thus the limsup inequality for $\ESMMb{\eps}{b}$ for $b>0$ is obtained.
 
It remains to handle the case for $M=\overline{I}$.
 Assume that $x_1 \in \Sigma$ is the right end point of $\overline{I}$.
 We first consider the case when $G(\xi^-_1) \leq G(\xi^+_1)$.
 Instead of \eqref{MME}, we have
\[
	\ESMM{\eps} \left(v^\delta_{\eps, 1}\right) \leq 2 G(\xi^+_1) + G(\xi^-_1) + 3 \beta_1 \delta.
\] 
If there is no other point of $\Sigma$ on $\partial I$, arguing in the same way we obtain the desired limsup inequality by the same construction of $w_\eps$.
 If $G(\xi^-_1) > G(\xi^+_1)$, then we modify the definition of $v^\delta_{\eps, 1}$ by
\begin{equation*}
	\overline{v}^\delta_{\eps, 1}(x) := \left \{
	\begin{array}{ll}
	v^\delta_\eps(x, x_1, \xi^+_i) & \text{of}\quad x \in (x_1- \eta_i, x_1] \\
	v^\delta_\eps(x, x_1 + \eta^+_1 + \eta^-_1, \xi^-_i) & \text{otherwise.}
\end{array}
	\right.
\end{equation*}
The remaining argument is similar.
 Symmetric argument yields the limsup inequality in the case that $\Sigma$ has the left end point of $\overline{I}$.
\end{proof}
We next prove the compactness.
\begin{proof}[Proof of Theorem \ref{C1}]
As in the proof of Theorem \ref{G1} (i), we see that
\[
	\sup_j \int_M \left|G\left( v_{\eps_j} \right)_x\right| < \infty.
\]
By (F2), we see that
\[
	G(v) \leq \sqrt{F(v)} |v-1| \leq \frac{F(v)}{2} + \frac{(v-1)^2}{2}, \quad
	v \in \mathbf{R}.
\]
By (F2'), we see
\[
	F(v) \geq c_0 (v-1)^2 -c'_1
\]
so
\[
	G(v) \leq C' F(v)
\]
for $v$ such that $|v-1|$ is sufficiently large $v$ with some content $C'>0$.
 Since
\[
	\frac{1}{\eps_j} \int_M F \left(v_{\eps_j} \right)
\]
is bounded, so is $\int_M G\left( v_{\eps_j} \right)$.
 We set $u^{\eps_j}=G\left( v_{\eps_j} \right)$ and observe that $\mathrm{TV}\left( u^{\eps_j} \right)$ is bounded and $\left\| u^{\eps_j} \right\|_{L^1}$ is bounded.
 Since
\[
	\left\| f - f_{\mathrm{av}} \right\|_\infty \leq \| f_x \|_{L^1},
\]
where $f_{\mathrm{av}}$ is the average of $f$ over $I$, it follows that
\[
	 \| f \|_\infty \leq \| f_x \|_{L^1} + \| f_x \|_{L^1}/|I|.
\]
This interpolation inequality yields a bound for $\left\| u^{\eps_j} \right\|_\infty$.
 Applying Lemma \ref{L2}, there is a subsequence $U^{\eps_k}$ converges to $V$ uniformly, {where $U^{\eps_k}$ is the unfolding of $u^{\eps_k}$}.
 Since we may assume that $x^{\eps_k}$, {the inverse of arc-length of $u^{\eps_k}$,} converges to $\overline{x}$ uniformly in $M$ by taking a subsequence,
 applying Theorem~\ref{T4} yields that
\[
	\limsupstar_{k\to\infty} u^{\eps_k} (x) 
	> \liminfstar_{k\to\infty} u^{\eps_k} (x), \quad
	x \in \Sigma 
\]
at most a countable set $\Sigma$.
 Since $v_{\eps_k}\to 1$ a.e.\ by taking a subsequence,
 we see that $\liminfstar u^{\eps_k} (x)=0$ for all $x \in M$.
 This implies that ${v_{\eps_k}}$ satisfies all  assumptions on a sequence $\{g_j\}$ of the compactness lemma (Lemma~\ref{SC}) with $S=\Sigma$.
 Then by Lemma~\ref{SC}, we conclude that $v_{\eps_k} \graphto \Xi$ with some $\Xi \in \mathcal{A}_0$.
\end{proof}

\section{Singular limit of the Kobayashi--Warren--Carter energy} \label{sec:5} 

In this section, we shall study the Gamma limit of the Kobayashi--Warren--Carter energy.

We first derive an inequality for lower semicontinuity.
 Assume that $M$ is either $\overline{I}$ or $\mathbf{T}$.
 Assume that
\begin{enumerate}
\item[(C1)] $v_\eps\graphto\Xi$, $u_\eps \to u$ in $L^1(\ringM)$ as $\eps \to 0$, where $v_\eps \in C(M)$, $u_\eps \in L^1$.
\end{enumerate}

For the limits, we assume that
\begin{enumerate}
\item[(C2)] $\Xi\in\mathcal{A}_0$, that is, there is a countable set $\Sigma = \{x_i\}^\infty_{i=1} \subset M$ such that $\Xi(x)=\{1\}$ for $x \notin \Sigma$ and $\Sigma(x_i)=\left[ \xi^-_i, \xi^+_i \right] \ni 1$ with $\xi^-_i < \xi^+_i$ for  $x_i \in \Sigma$.
 Moreover, $\sum^\infty_{i=1} G\left(\xi^-_i\right) < \infty$.

\item[(C3)] $u \in BV\bigl(\ringM\setminus\Sigma_0\bigr)$, where $\Sigma_0 = \left\{ x_i \in \Sigma \bigm| \xi^-_i =0 \right\}$.
 (Since $\sum^\infty_{i=1} G\left( \xi^-_i \right)<\infty$, the set $\Sigma_0$ is a finite set.)
\end{enumerate}
We define a weighted total variation
\[
	\int_{\ringM} v^2_\eps \left| \odfrac{u_\eps}{x} \right| :=
	\sup \left\{ \int_{\ringM} \varphi_x u_\eps \id x \Bigm| 
	\left| \varphi(x) \right| \leq v^2_\eps(x),\ 
	\varphi \in C^1_c (\ringM) \right\},
\]
where $C^1_c (M)$ is the space of all $C^1$ functions in $\ringM$ with compact support in $\ringM$.
 For $u \in BV\bigl(\ringM\setminus\Sigma_0\bigr)$, let $J_u$ denote the set of jump discontinuities of $u$.
 In other words,
\[
	\textstyle
	J_u = \left\{ x \in \ringM\setminus \Sigma_0 \bigm| 
	d(x) = \left| u(x+0) - u(x-0) \right| >0 \right\},
\]
where $u(x \pm 0)$ is the trace from right ($+$) and left ($-$).
It is at most a countable set.
\begin{lemma} \label{Lemma6} 
Assume that (C1) -- (C3).
 Then
\[
	\int_{\ringM\setminus\left(J_u\cap\Sigma\right)} |u_x|
	+ \sum_{x \in \Sigma'} d_i \left| \xi^-_i \right|^2
	\leq \liminf_{\eps \to 0} \int_{\ringM} v^2_\eps \abs{\odfrac{u_\eps}{x}},
\]
where $d_i = d(x_i) \geq 0$, $\Sigma'=\Sigma\setminus\Sigma_0$.
\end{lemma}
\begin{proof}
It suffices to prove that for any $\delta\in(0,1)$,
\[
	(1-\delta)^2 \int_{\ringM\setminus\left(J_u\cap\Sigma_{1-\delta}\right)} |u_x|
	+ \sum_{x_i \in \Sigma'_{1-\delta}} d_i \left| \xi^-_i \right|^2
	\leq \liminf_{\eps \to 0} \int_{\ringM} v^2_\eps \abs{\odfrac{u_\eps}{x}},
\]
where
\[
	\textstyle
	\Sigma_{1-\delta} = \left\{ x_i \in \Sigma \bigm| \xi^-_i < 1-\delta \right\}, \quad
	\Sigma'_{1-\delta} = \Sigma_{1-\delta} \setminus \Sigma_0.
\]
By this notation $\Sigma_1=\Sigma$.
 Note that the set $\Sigma_{1-\delta}$ is a finite set for $\delta>0$ since $\sum^\infty_{i=1} G\left(\xi^-_i\right)<\infty$.
 
Since $\Sigma_0$ is a finite set and 
\[
	\int_{\ringM} v^2_\eps \abs{\odfrac{u_\eps}{x}}
	\geq \int_{\ringM\setminus\Sigma_0} v^2_\eps \abs{\odfrac{u_\eps}{x}},
\]
it suffices to prove that for each interval $\{M_j\}^m_{j=1}$, which is a connected component of $\ringM\setminus\Sigma_0$ the inequality
\[
	(1-\delta)^2 \int_{M_j\setminus\left(J_u\cap\Sigma_{1-\delta}\right)} |u_x|
	+ \sum_{\substack{x_i \in \Sigma'_{1-\delta}\\ x_i \in M_j}} d_i \left| \xi^-_i \right|^2
	\leq \liminf_{\eps \to 0} \int_{M_j} v^2_\eps \abs{\odfrac{u_\eps}{x}}.
\]
Thus we may assume that $\Sigma_0=\emptyset$.

We consider $\delta_1$-open neighborhood of $\Sigma_{1-\delta}$, that is,
\[
	\textstyle
	X_{\delta_1} = \left\{ x \in \ringM \bigm| \dist\left(x,\Sigma_{1-\delta}\right) < \delta_1 \right\}
\]
and observe that
\[
	\int_{\ringM} v^2_\eps \left|\odfrac{u_\eps}{x}\right|
	\geq \int_{{\ringM} \setminus\overline{X}_{\delta_1}} v^2_\eps \left|\odfrac{u_\eps}{x}\right|
	+ \int_{X_{\delta_1}} v^2_\eps \left|\odfrac{u_\eps}{x}\right|.
\]
We may assume that $X_{\delta_1}$ consist of disjoint interval $B_{\delta_1}(x_i) = \left\{ x \in \ringM \bigm| |x-x_i| < \delta_1\right\}$, $x_i \in \Sigma_{1-\delta}$ by taking $\delta_1$ small.
 Since $v_\eps\graphto\Xi$, for sufficiently small $\eps$ we observe that
\begin{align*}
	v_\eps &\geq 1-\delta-\delta_1 &\text{in}\quad &\ringM\setminus \overline{X}_{\delta_1}, \\
	v_\eps &\geq \xi^-_i - \delta_1 &\text{in}\quad &B_{\delta_1}(x_i),\ \textstyle x_i \in \Sigma_{1-\delta}.
\end{align*}
We thus conclude that
\begin{align*}
	&\liminf_{\eps\to 0} \int_{\ringM} v_\eps^2 \abs{\odfrac{u_\eps}{x}} \\
	&\geq \liminf_{\eps\to 0} \left\{ (1-\delta-\delta_1)^2 \int_{\ringM\setminus\overline{X}_{\delta_1}} \abs{\odfrac{u_\eps}{x}}
	+ \sum_{x_i \in \Sigma_{1-\delta}} \left(\xi^-_i - \delta_1\right)^2 \int_{B_{\delta_1}(x_i)} \abs{\odfrac{u_\eps}{x}} \right\} \\
	&\geq (1-\delta-\delta_1)^2 \int_{\ringM\setminus\overline{X}_{\delta_1}} |u_x|
	+ \sum_{x_i \in \Sigma_{1-\delta}} \left(\xi^-_i - \delta_1\right)^2 \int_{B_{\delta_1}(x_i)} |u_x|
\end{align*}
by lower semicontinuity of $\mathrm{TV}(\cdot)$ with respect to $L^1$-convergence.
 The second term of the right-hand side is estimated from below by
\[
	\sum_{x_i \in \Sigma_{1-\delta}} \left(\xi^-_i - \delta_1\right)^2 d_i.
\]
Note that $\delta_2 < \delta_1$ implies $\ringM \backslash \overline{X}_{\delta_1} \subset \ringM \backslash \overline{X}_{\delta_2}$. 
Sending $\delta_1\to 0$ yields
\[
	(1-\delta)^2 \int_{\ringM\setminus\Sigma_{1-\delta}} \abs{u_x}
	+ \sum_{x_i \in \Sigma_{1-\delta}} d_i \abs{\xi_i^-}^2
	\leq \liminf_{\eps\to 0} \int_{\ringM} v^2_\eps \abs{\odfrac{u_\eps}{x}}.
\]
Replacement of $\ringM\setminus\Sigma_{1-\delta}$ by $\ringM\setminus\left( J_u\cap\Sigma_{1-\delta} \right)$ is rather trivial because outside $J_u$ the set $\Sigma_{1-\delta}$ has measure zero with respect to the measure $\abs{u_x}$. 
\end{proof}
We are now in position to give a proof for the Gamma limit of the Kobayashi--Warren--Carter energy.
\begin{proof}[Proof of Theorem~\ref{G2}]
	\vspace{-1em}
\begin{enumerate}[leftmargin=2em,topsep=0em]
\item[(i)] (liminf inequality)
 We may assume that
 \[
 \liminf_{\eps \to 0} \EKWC{\eps} (u_\eps,v_\eps) < \infty.
 \]
 By Theorem~\ref{G1}(i), we see that the limit $\Xi$ satisfies (C2). 
 Let $\Omega$ be an open set such that $\overline{\Omega}$ is compact and contained in $\ringM\setminus\Sigma_0$.
 Assume that
\[
	\int_{\ringM} v^2_\eps \abs{\odfrac{u_\eps}{x}}
\]
is bounded.
 Since $c:=\min_{x \in \Sigma'} \Xi(x) > 0$ and $v_\eps\xrightarrow{g} \Xi$, we set that $v_\eps \geq c/2 > 0$ on $\overline{\Omega}$ for sufficiently small $\eps > 0$.
 Thus $\displaystyle\int_\Omega \abs{\odfrac{u_\eps}{x}}$ is bounded.
 This implies that the limit $u \in BV(\Omega)$.
 We now conclude that $u$ satisfies (C3). 
 
Applying Theorem~\ref{G1} for $\ESMM{\eps}$ and Lemma~\ref{Lemma6} for $\displaystyle \int v^2_\eps \abs{\odfrac{u_{\eps}}{x}}$, we see that
\[
	\sigma \int_{\ringM\setminus(J_{u}\cap\Sigma)} |u_x| + \sigma \sum_{x \in\Sigma'} d_i |\xi^-_i|^2
	+ \ESMM{0}(\Xi, M) \leq \liminf_{\eps\to 0} \EKWC{\eps}(u_\eps, v_\eps).
\]
 The second term in the left-hand side equals $\sigma \sum_{x \in \Sigma} d_i |\xi^-_i|^2$ since $\xi^-_i = 0$ on $\Sigma_0$.
 Thus the left-hand side equals $\EKWC{0}(u, \Xi, M)$.
 The proof of liminf inequality is now complete. 
\item[(ii)] (limsup inequality)
 We take $u_\eps = u$.
 We notice that Theorem~\ref{G1} extends to the case when $\ESMMb{0}{b}(\Xi, M)$, $\ESMMb{\eps}{b}(v)$ are replaced by
\begin{align*}
	\ESMMb{0}{\{b_\ell\}}(\Xi, M) &:= \ESMM{0}(\Xi, M) + \sum^\infty_{\ell=1} b_\ell \left(\min\Xi(a_\ell)\right)^2, \\
	\ESMMb{\eps}{\{b_\ell\}}(v_\eps) &:= \ESMM{\eps}(v_\eps) + \sum^\infty_{\ell=1} b_\ell \left(v_\eps(a_\ell)\right)^2,
\end{align*}
where we assume that $\sum^\infty_{\ell=1}b_\ell<\infty$ with $b_\ell \geq 0$ and $a_\ell \in \ringM$ for $\ell=1, 2, \ldots, m$.
 Let $\{a_\ell\}$ denote the jump discontinuity of $u$, that is, $J_u=\{a_\ell\}$.
 Let $b_\ell$ denote $\sigma$ times the jump $d_\ell = \left| u(a_\ell+0) - u(a_\ell-0)\right|$, that is, $b_\ell=\sigma d_\ell$.
 Note that $\sum b_\ell<\infty$.
 By Theorem~\ref{G1}(ii) for $\ESMMb{\eps}{\{b_\ell\}}$, we see that there exist $w_\eps\xrightarrow{g}\Xi$ such that
\begin{equation} \label{G3}
	\ESMMb{0}{\{b_\ell\}}(\Xi, M) = \lim_{\eps \to 0} \ESMMb{\eps}{\{b_\ell\}}(w_\eps).
\end{equation}
We notice that
\begin{align*}
	\EKWC{\eps}(u, w_\eps) 
	&= \sigma \int_{\ringM} w^2_\eps |u_x| + \ESMM{\eps}(w_\eps) \\
	&= \sigma \int_{\ringM\setminus \Sigma} w^2_\eps |u_x| + \sum^\infty_{\ell=1} b_\ell w_\eps(a_\ell)^2 + \ESMM{\eps}(w_\eps) \\
&= \sigma \int_{\ringM\setminus \Sigma} w^2_\eps |u_x| + \ESMMb{\eps}{\{b_\ell\}}(w_\eps).
\end{align*}
By construction $w_\eps$ is bounded and $w_\eps \to 1$ almost everywhere in the sense of all continuous measure.
Since $w_\eps^2-1$ tends to zero for all $x$ outside $\Sigma$ and it is bounded, 
the first term in the right-hand side converges to $\sigma \int_{M\setminus \Sigma} \abs{u_x}$ by a bounded convergence theorem.
 The convergence \eqref{G3} yields the desired result.
\end{enumerate}
\end{proof}
\section{Acknowledgement}
The work of the first author was partly supported by Japan~Society~for~the~Promotion~of~Science (JSPS) through the grants KAKENHI No.~26220702, No.~19H00639, No.~18H05323, No.~17H01091, No.~16H03948 and by Arithmer,~Inc. through collaborative grant. The work of the second author was partly supported by the Leading~Garduate~Program ``Frontiers of Mathematical Sciences and Physics,'' JSPS.


\begin{thebibliography}{99}
%
%

\bibitem{AF}
Aubin,~J.-P., Frankowska,~H.: 
\newblock Set-valued analysis. 
\newblock Modern Birkh\"auser Classics. Birkhh\"auser Boston, Inc., Boston, MA (2009)

\bibitem{AT1}
Ambrosio,~L., Tortorelli,~V.~M.: 
\newblock Approximation of functionals depending on jumps by elliptic functionals via $\Gamma$-convergence. 
\newblock Comm.~Pure Appl.~Math. 43, 999--1036 (1990)

\bibitem{AT2}
Ambrosio,~L., Tortorelli,~V.~M.: 
\newblock On the approximation of free discontinuity problems. 
\newblock Boll.~Un.~Mat.~Ital.~B (7) 6, 105--123 (1992)

\bibitem{BLM}
Bonnivard,~M., Lemenant,~A., Millot,~V.: 
\newblock On a phase field approximation of the planar Steiner problem: existence, regularity, and asymptotic of minimizers. 
\newblock Interfaces Free Bound. 20, 69--106 (2018)

\bibitem{Br}
Braides, A.:
\newblock $\Gamma$-convergence for beginners.
\newblock Oxford University Press, Oxford (2002)

\bibitem{BK}
Bronsard,~L., Kohn,~R.~V.,: 
\newblock Motion by mean curvature as the singular limit of Ginzburg-Landau dynamics.
\newblock J.~Differential Equations 90, 211--237 (1991)

\bibitem{XC}
Chen, X.: 
\newblock Generation and propagation of interfaces for reaction-diffusion equations. 
\newblock J.~Differential Equations 96, 116--141 (1992)

\bibitem{DS}
de~Mottoni,~P., Schatzman,~M.: 
\newblock Geometrical evolution of developed interfaces. 
\newblock Trans.~Amer.~Math.~Soc. 347, 1533--1589 (1995)

\bibitem{ESS}
Evans,~L.~C., Soner,~H.~M., Souganidis,~P.~E.: 
\newblock Phase transitions and generalized motion by mean curvature. 
\newblock Comm.~Pure Appl.~Math. 45, 1097--1123 (1992)

\bibitem{FLS}
Francfort,~G.~A., Le,~N.~Q., Serfaty,~S.: 
\newblock Critical points of Ambrosio--Tortorelli converge to critical points of Mumford--Shah in the one-dimensional Dirichlet case. 
\newblock ESAIM Control Optim.~Calc.~Var., 15, 576--598 (2009)

\bibitem{FM} 
Francfort,~G.~A., Marigo,~J.-J.: 
\newblock Revisiting brittle fracture as an energy minimization problem. 
\newblock J.~Mech.~Phys.~Solids 46, 1319--1342 (1998)

\bibitem{FL}
Fonseca,~I., Liu,~P.:
\newblock The Weighted Ambrosio--Tortorelli Approximation Scheme.
\newblock SIAM J.~Math.~Anal., 49(6), 4491--4520 (2017)

\bibitem{G} 
Giacomini,~A.: 
\newblock Ambrosio-Tortorelli approximation of quasi-static evolution of brittle fractures. 
\newblock Calc.~Var.~Partial Differential Equations 22, 129--172 (2005)

\bibitem{Giga}
Giga,~Y.:
\newblock Surface evolution equations: a level set approach. 
\newblock Birkh\"{a}user, Basel (2006)

\bibitem{HT} 
Hutchinson,~J.~E., Tonegawa,~Y.: 
\newblock Convergence of phase interfaces in the van der Waals-Cahn-Hilliard theory. 
\newblock Calc.~Var.~Partial Differential Equations 10, 49--84 (2000)

\bibitem{IKY}
Ito,~A., Kenmochi,~N., Yamazaki,~N.: 
\newblock A phase-field model of grain boundary motion. 
\newblock Appl.~Math. 53, 433--454 (2008)

\bibitem{KG} 
Kobayashi,~R., Giga,~Y.: 
\newblock Equations with singular diffusivity. 
\newblock J.~Statist.~Phys. 95, 1187--1220 (1999)

\bibitem{KWC98}
Kobayashi,~R., Warren,~J.~A., Carter,~W.~C.: 
\newblock Modeling grain boundaries using a phase field technique. 
\newblock Hokkaido University Preprint Series in Mathematics \#422 (1998)

\bibitem{KWC00}
Kobayashi,~R., Warren,~J.~A., Carter,~W.~C.:
\newblock A continuum model of grain boundaries.
\newblock Physica D: Nonlinear Phenomena, 140(1--2), 141--150 (2000)


\bibitem{KWC}
Kobayashi,~R., Warren,~J.~A., Carter,~W.~C.: 
\newblock Grain boundary model and singular diffusivity:
\newblock In: Free boundary problems: theory and applications,
\newblock GAKUTO Internat.~Ser.~Math.~Sci.~Appl. 14, 283--294, Gakk\={o}tosho, Tokyo (2000)

\bibitem{KS} 
Kohn,~R., Sternberg,~P.: 
\newblock Local minimisers and singular perturbations.
\newblock Proc.~Roy.~Soc.~Edinburgh Sect.~A 111, 69--84 (1989)

\bibitem{LS} 
Lemenant,~A., Santambrogio,~F.: 
\newblock A Modica--Mortola approximation for the Steiner problem. 
\newblock C.~R.~Math.~Acad.~Sci. Paris 352, 451--454 (2014)

\bibitem{Mo} 
Modica,~L.: 
\newblock The gradient theory of phase transitions and the minimal interface criterion. 
\newblock Arch.~Rational Mech.~Anal. 98, 123--142 (1987)

\bibitem{MM1} 
Modica,~L., Mortola,~S.: 
\newblock Il limite nella $\Gamma$-convergenza di una famiglia di funzionali ellittici. 
\newblock Boll.~Un.~Mat.~Ital.~A (5), 14, 526--529 (1977)

\bibitem{MM2} 
Modica,~L., Mortola,~S.: 
\newblock Un esempio di $\Gamma^-$-convergenza. 
\newblock Boll.~Un.~Mat.~Ital.~B (5), 14, 285--299 (1977)

\bibitem{MS} 
Mumford,~D., Shah,~J.: 
\newblock Optimal approximations by piecewise smooth functions and associated variational problems. 
\newblock Comm.~Pure Appl.~Math. 42, 577--685 (1989)

\bibitem{MR3268865}
Moll,~S., Shirakawa,~K.:
\newblock Existence of solutions to the {K}obayashi--{W}arren--{C}arter system.
\newblock Calc.~Var.~Partial Differential Equations, 51(3-4):621--656 (2014)

\bibitem{MR3670006}
Moll,~S., Shirakawa,~K., Watanabe,~H.:
\newblock Energy dissipative solutions to the {K}obayashi--{W}arren--{C}arter
  system.
\newblock Nonlinearity, 30(7):2752--2784 (2017)

\bibitem{mswD2020}
Moll,~S., Shirakawa,~K., Watanabe,~H.:
\newblock Kobayashi--Warren--Carter type systems with nonhomogeneous Dirichlet boundary data for crystalline orientation.
\newblock In preparation 

\bibitem{MR3203495}
Watanabe,~H., Shirakawa,~K.:
\newblock Qualitative properties of a one-dimensional phase-field system
  associated with grain boundary.
\newblock In Nonlinear analysis in interdisciplinary
  sciences---modellings, theory and simulations, volume~36 of GAKUTO
  Internat.~Ser.~Math.~Sci.~Appl., pages 301--328. Gakk\=otosho, Tokyo, 2013.

\bibitem{MR3082861}
Shirakawa,~K., Watanabe,~H.:
\newblock Energy-dissipative solution to a one-dimensional phase field model of
  grain boundary motion.
\newblock Discrete Contin.~Dyn.~Syst.~Ser.~S, 7(1):139--159 (2014)

\bibitem{SWY}
Shirakawa,~K., Watanabe,~H., Yamazaki,~N.: 
\newblock Solvability of one-dimensional phase field systems associated with grain boundary motion. 
\newblock Math.~Ann.~356, 301--330 (2013)

\bibitem{St} 
Sternberg,~P.: 
The effect of a singular perturbation on nonconvex variational problems. 
Arch.~Rational Mech.~Anal. 101, 209--260 (1988)
\end{thebibliography}
\end{document}